\documentclass[a4paper]{amsart}
\title{Almost positive links are strongly quasipositive}
\author{Peter Feller}
\address{ETH Zurich, R\"amistrasse 101, 8092 Zurich, Switzerland}
\email{\myemail{peter.feller@math.ch}}
\urladdr{\url{https://people.math.ethz.ch/~pfeller/}}
\author{Lukas Lewark}
\address{University of Bern, Mathematical Institute, Alpeneggstr. 22, 3012 Bern, Switzerland}
\email{\myemail{lukas@lewark.de}}
\urladdr{\url{http://www.lewark.de/lukas/}}
\keywords{Quasipositive links, almost positive links, canonical Seifert surfaces, slice-torus invariants}
\subjclass[2010]{57M25}
\author{Andrew Lobb}
\address{Mathematical Sciences, Durham University, UK}
\email{\myemail{andrew.lobb@durham.ac.uk}}
\urladdr{\url{http://www.maths.dur.ac.uk/users/andrew.lobb/}}
\usepackage{thmtools,enumerate,graphicx,xcolor,enumitem,afterpage,etex,float,ucs,amsmath,amssymb,amscd,tabu,diagbox,sseq,amssymb,epsfig,psfrag,microtype,thm-restate,url}
\usepackage{pinlabel}
\usepackage{mathtools}
\usepackage[latin1]{inputenc}
\usepackage[all]{xy}
\usepackage{amsthm}
\usepackage{hyperref}
\usepackage[nameinlink]{cleveref}
\let\cref\Cref
\crefname{subsection}{subsection}{subsections}
\Crefname{subsection}{Subsection}{Subsections}
\Crefname{enumi}{}{}
\definecolor{darkblue}{RGB}{0,0,96}
\definecolor{gray}{RGB}{127,127,127}
\definecolor{darkred}{RGB}{160,0,0}
\definecolor{lightyellow}{RGB}{255,255,128}
\hypersetup{colorlinks={true},linkcolor={black},citecolor={black},filecolor={black},urlcolor={black},pdfauthor={\authors},pdftitle={\shorttitle}}
\newcommand{\myemail}[1]{\href{mailto:#1}{#1}}
\newcommand{\qua}{\hskip 0.4em \ignorespaces}
\def\arxiv#1{\relax\ifhmode\unskip\qua\fi
\href{http://arxiv.org/abs/#1}%
{\tt arXiv:\penalty -100\unskip#1}}

\def\MR#1{\relax\ifhmode\unskip\qua\fi
\href{http://www.ams.org/mathscinet-getitem?mr=#1}{\tt MR#1}}
\def\xox#1{\csname xx#1\endcsname}

\newtheorem{theoremA}{Theorem}
\newtheorem{corollaryA}[theoremA]{Corollary}

\declaretheorem[numberwithin=section]{lemma}

\newtheorem{proposition}[lemma]{Proposition}

\Crefname{theoremA}{Theorem}{Theorem}
\Crefname{corollaryA}{Corollary}{Corollary}
\theoremstyle{definition}
\newtheorem*{examp}{Example}
\newtheorem*{rmk}{Remark}
\newtheorem{definition}[lemma]{Definition}
\newtheorem*{Def}{Definition}
\newtheorem*{construction}{Construction}
\newtheorem{remark}[lemma]{Remark}

\DeclareMathAlphabet{\mathpzc}{OT1}{pzc}{m}{it}

\newcommand{\Z}{\mathbb{Z}}

\newcommand{\R}{\mathbb{R}}

\begin{document}
\thispagestyle{empty}
\subjclass[2010]{57M25}
\begin{abstract}
We prove that any link admitting a diagram with a single negative crossing is strongly quasipositive.
This answers a question of Stoimenow's in the (strong) positive.  As a second main result, we give simple and complete characterizations of link diagrams with quasipositive canonical surface (the surface produced by Seifert's algorithm). As applications, we determine which prime knots up to 13 crossings are strongly quasipositive, and we confirm the following conjecture for knots that have a canonical surface realizing their genus:
a knot is strongly quasipositive if and only if the Bennequin inequality is an equality.

\end{abstract}
\maketitle
\section*{Introduction}
Notions of quasipositivity for links and surfaces were introduced and explored by Rudolph in a series of papers
(cited in the text).
Their study is motivated, for example, by connections with complex algebraic plane curves~\cite{rud,MR1836094} and relationships to contact geometry~\cite{MR2562830,MR2646650}.

Quasipositive links, strongly quasipositive links, and quasipositive Seifert surfaces are usually defined in terms of braids.
In this paper, however, our focus lies more on geometry and less on braids, and so we omit the original
definitions in favor of the following characterizations:
a Seifert surface is called \emph{quasipositive} if it
is an incompressible subsurface  of the fiber surface of a positive torus link
(\emph{incompressible} meaning that the map induced by inclusion on the fundamental group is injective)
and \emph{strongly quasipositive links} are precisely those links that arise as the boundary of a quasipositive Seifert surface.
That these characterizations are equivalent to the original definitions is due to Rudolph~\cite{qpossubsurfaces}.
We are not concerned with (non-strongly) quasipositive links in this text.

Quasipositive Seifert surfaces are of maximal Euler characteristic; not just among Seifert surfaces of the given link, but even
among smooth slice surfaces~\cite{thom,rudolph}.

\subsection*{Main results}
Links that admit a positive diagram, in other words a diagram without negative crossings, are known as \emph{positive links}.  Positive links are strongly quasipositive \cite{posisqpos,MR1772843}.
Our first main result generalizes this to \emph{almost positive links}---links admitting an almost positive diagram, in other words a diagram with a single negative crossing.
This gives a positive answer to a question of Stoimenow~\cite[Question~4]{MR2159224}.
\begin{theoremA}
	\label{thm:1}
Almost positive links are strongly quasipositive.
\end{theoremA}
Note that the hypothesis cannot be weakened further (at least in the most obvious way), since links admitting diagrams with two negative crossings need not even be quasipositive (for example, the figure eight knot).

Almost positive links have been studied before they were given this name, and their similarity in many respects to positive links
has been observed. For example, Cromwell showed that almost positive links have Conway polynomials with
non-negative coefficients \cite{cromwell-homogeneous}; Przytycki and Taniyama proved they have negative signature \cite{MR2647054};
Stoimenow showed that non-trivial almost positive links are chiral and non-slice \cite{MR1825928};
and Tagami proved that the 3--genus, 4--genus, and $s/2$ (for $s$ the Rasmussen invariant) of almost positive knots agree \cite{tagami}.
\cref{thm:1} can be seen in this context.  In particular, \cref{thm:1} recovers the last result:
the slice-Bennequin inequality implies that for any strongly quasipositive knot (and thus for any almost positive knot)
the 3--genus, 4--genus, and all slice-torus invariants agree \cite{rudolph,livingston,lew2}.
This also proves chirality and non-sliceness for non-trivial almost positive knots.
Here, a \emph{slice-torus invariant} \cite{livingston,lew2} is a homomorphism $y$ from the smooth concordance group to $\R$
such that for all knots~$K$, $y(K)$ is a lower bound for the 4--genus of $K$, and for positive torus knots~$K$,
$y(K)$ is equal to the 4--genus of $K$. Examples of such $y$ include $\tau$ from knot Floer homology, and $s/2$.

To prove \cref{thm:1}, we explicitly exhibit quasipositive Seifert surfaces for all almost positive links.
For a certain type of almost positive diagram (which will later be referred to as type~I),
we prove that in fact the \emph{canonical surface} is quasipositive (the canonical surface is that produced from the diagram by Seifert's algorithm).
Canonical surfaces have been studied extensively; it is for example a classical result that the canonical surfaces of alternating diagrams are genus-minimizing \cite{MR0099664,MR0099665}, which generalizes to homogeneous diagrams \cite{cromwell-homogeneous},
and has recently been scrutinized further \cite{MR2443760,MR3499519}.
In this light, our proof of \cref{thm:1} naturally begs the question:
\emph{which canonical surfaces are quasipositive?}
The complete answer to this question forms our second main result:
namely a criterion in terms of the Seifert graph,
which is combinatorial and algorithmic.%
\begin{restatable}{theoremA}{thmB}
\label{thm:2}
A canonical surface is quasipositive if and only if all cycles of
its Seifert graph have strictly positive total weight.
\end{restatable}
Here, the \emph{Seifert graph} $\Gamma(D)$ of a diagram $D$ has the Seifert circles as vertex set,
and one edge between $k$ and $k'$ for each crossing connecting the Seifert circles $k$ and~$k'$.
It is a bipartite graph, possibly with multiple edges between two vertices.
Its edges carry a \emph{weight} of $\pm 1$ corresponding to the sign of the crossing.
The \emph{total weight} of a cycle is understood as the sum of the weights of the cycle's edges.
The reader will find more details on Seifert graphs at the beginning of \cref{sec:almost1}.

\subsection*{Applications of \cref{thm:2}}
\cref{thm:2} implies a purely geometric criterion for quasipositivity of a canonical Seifert surface, which we state as the following corollary.
\begin{restatable}{corollaryA}{corC}
\label{cor:thm2viagraph}
Let $\Sigma$ be a Seifert surface that is isotopic to a canonical surface.
Then~$\Sigma$ is quasipositive if and only if
every unknot contained in $\Sigma$
bounds a disk in~$\Sigma$ or has negative induced framing by $\Sigma$.
\end{restatable}
\cref{cor:thm2viagraph} does not generalize
to non-canonical Seifert surfaces; in fact, there exist
non-quasipositive Seifert surfaces $\Sigma$ such that all incompressible annuli of $\Sigma$ 
are quasipositive Seifert surfaces \cite{MR2823096}.

Next we observe the following criterion for quasipositivity. This allows, see the example below, to determine the strong-quasipositivity status of all prime knots up to 13 crossings, in particular recovering the recently completed calculation~\cite{12} of the strong-quasipositivity status of prime knots up to 12 crossings.
Throughout this subsection, let $y$ denote a slice-torus invariant.
\begin{restatable}{theoremA}{thmD}
\label{prop:s=cang=>sqp}  If $K$ is a knot with a canonical surface $\Sigma$ such that $y(K)=\mathrm{genus}(\Sigma)$, then $\Sigma$ is a quasipositive Seifert surface; in particular, $K$ is a strongly quasipositive knot.
\end{restatable}
Recall that the Bennequin inequality states $\frac{\mathrm{sl}(K)+1}{2}\leq g(K)$, where $\mathrm{sl}(K)$ is defined in either of the following two equivalent ways; see~\cite{benn}:
\begin{align*}
\mathrm{sl}(K)&\coloneqq\max\{\mathrm{sl}(T)\mid T\text{ is a transverse representative of $K$}\}\\
\mathrm{sl}(K)&\coloneqq \max\{\mathrm{writhe}(\beta)-n\mid\beta\text{ is an $n$-braid with closure } K\}.\end{align*}
It is a conjecture (popularized by Hedden, Etnyre and Van Horn-Morris amongst others)
that the Bennequin inequality is an equality if and only if $K$ is strongly quasipositive; compare also \cite{12}. As a consequence of \Cref{prop:s=cang=>sqp}, we confirm this conjecture for knots with canonical genus $\tilde{g}$ (the minimum genus of a canonical surface) equal to the genus.
\begin{restatable}{corollaryA}{corE}
\label{cor:g<=>cang=>sqp<=>s=g}
Let $K$ be a knot with $\tilde{g}(K) = g(K)$, i.e.\ a knot for which the genus $g(K)$ is realized by a canonical surface $\Sigma$. The following are equivalent:
\begin{enumerate}
\item $\Sigma$ is quasipositive,
\item $K$ is strongly quasipositive,
\item for $K$ the Bennequin-inequality is an equality, and
\item $y(K)=g(K)$.
\end{enumerate}
\end{restatable}
Note that for $K$ a fibered knot with fiber surface $\Sigma$, the conditions of (1)--(4) in \cref{cor:g<=>cang=>sqp<=>s=g}
are also equivalent \cite{MR2646650}.

Applied to the canonical surface $\Sigma(p_1,\ldots, p_{2n+1})$ of genus $n$ of the $P(p_1,\ldots, p_{2n+1})$ pretzel knot with all $p_1, \ldots, p_n\in\Z$ odd, \cref{thm:2} immediately yields that this surface is 
quasipositive if and only if $p_i + p_j<0$ for all $1\leq i < j \leq n$,
recovering a result of Rudolph's~\cite{rudolph,MR1840734}.
Moreover, since $\Sigma(p_1,\ldots, p_{2n+1})$ is genus-minimizing \cite[Theorem~3.2]{MR823442}, \cref{cor:g<=>cang=>sqp<=>s=g} now implies the following.
\begin{corollaryA}
The $P(p_1,\ldots, p_{2n+1})$ pretzel knot with all $p_1, \ldots, p_n$ odd
is strongly quasipositive if and only if $p_i + p_j<0$ for all $1\leq i < j \leq n$.\qed
\end{corollaryA}

\begin{examp}
\label{ex:13cr}
We claim that if $K$ is a prime knot with crossing number $c(K) \leq 13$,
then $K$ is strongly quasipositive if and only if $y(K) = g(K)$.
The `only if' direction holds for all knots. To show the `if' direction,
we rely on Stoimenow's calculation \cite{stoitables} that for all prime knots $K$ with $c(K) \leq 13$,
it holds that $\tilde g(K) = m(K)$, where $m$ denotes Morton's lower bound \cite{MR809504} for $\tilde g$
coming from the Homflypt-polynomial. So, using \cref{prop:s=cang=>sqp},
it is enough to show that all prime knots with $c(K) \leq 13$ satisfy $\tau(K) < g(K)$ or $\tau(K) = m(K)$.
This is readily verified by a computer calculation.

Note that this criterion is algorithmic, and can be used in practice to determine
the strong quasipositivity status of a given
prime knot $K$ with $c(K) \leq 13$: simply calculate $\tau(K)$ and $m(K)$; $K$ is strongly quasipositive
if and only if $\tau(K) = m(K)$.

We do not know whether $y(K) = g(K)$ also implies strong quasipositivity for prime knots $K$ with $c(K) = 14$.
For $c(K) = 15$, we know it does not: the 2-twisted positive Whitehead double of the right-handed trefoil knot
is a prime 15-crossing knot (15n115646 in the table) with $\tau = g = 1$, which is not quasipositive since its Rasmussen invariant $s$ is 0 \cite{heddord}.
\end{examp}
\begin{rmk}\label{baader:rmk}
\cref{thm:2} also provides a new proof for Baader's theorem \cite{MR2168087} that a knot $K$ is positive
if and only if it is strongly quasipositive and homogeneous.

Here, following \cite{cromwell-homogeneous}, a knot is called \emph{homogeneous} if it admits a homogeneous diagram $D$;
and $D$ is called homogeneous if all edges within each block of the Seifert graph $\Gamma(D)$ carry the same weight;
where a block $B$ of a graph $G$ is either an isolated vertex of $G$,
or a maximal subgraph of $G$ with the property that $B$ is connected and $B\setminus v$ is connected for all vertices $v$ of $B$.

We leave it as an exercise in graph theory to show that a knot diagram $D$ is homogeneous if and only if
all edges within each cycle of $\Gamma(D)$ carry the same weight.

The `only if' direction of Baader's theorem is clear.
For the `if' direction, let a strongly quasipositive and homogeneous knot $K$ be given.
Let $D$ be a homogeneous diagram of $K$ that is also reduced (has no nugatory crossings).
For $\Sigma$ the canonical surface of $D$, we have $g(\Sigma) = g(K)$ \cite{cromwell-homogeneous}.
So, by \cref{cor:g<=>cang=>sqp<=>s=g}, $\Sigma$~is quasipositive.
We are going to show that $D$ is a positive diagram.  
Let an edge $e$ of $\Gamma(D)$ be given.
Since the crossing corresponding to $e$ is not nugatory, $e$ is contained in a cycle $C$ of $\Gamma(D)$.
Because $\Sigma$ is quasipositive, by \cref{thm:2}, the cycle $C$ has positive total weight.
Therefore, $C$ must contain at least one edge with weight $+1$.
But as discussed above, the homogeneity of $D$ implies that all edges within $C$ carry the same weight,
so in particular, $e$ has weight $+1$. This concludes the proof.
\end{rmk}
\subsection*{Approach to proofs}
In this subsection we give more details regarding the proofs of the main theorems.
We distinguish two types of diagrams.
\begin{Def}
	We say that $D$ is of \emph{type~I} if $D$ is positive, or if $D$ is almost positive and there is no positive crossing \emph{parallel} to the unique negative crossing
(in other words connecting the same pair of Seifert circles).
\end{Def}
\begin{Def}
	We say that $D$ is of \emph{type~II} if it is almost positive, and there is a positive crossing parallel to the unique negative crossing.
\end{Def}
We opted to include positive diagrams in type~I because they behave similarly in our constructions as almost positive diagrams of type~I;
for example our proof of \cref{thm:1} in \cref{sec:almost1} for links with a diagram of type~I
recovers Rudolph's result that positive links are strongly quasipositive.

The distinction between type~I and II is rather natural and has been made previously \cite{MR2159224,tagami,MR3499519}.
Stoimenow shows that each of the two types is realized by knots that do not admit diagrams of the other type;
he further shows that an almost positive diagram has a minimal genus canonical surface if and only if it is of type~I.
We strengthen this result from `minimal genus' to `quasipositive'.

Indeed, we show that if $D$ is of type~I, then the canonical surface $\Sigma(D)$ is quasipositive.
If $D$ is of type~II, we construct a quasipositive Seifert surface $\Sigma'(D)$ with $\partial\Sigma'(D) = \partial\Sigma(D)$ of genus one less than $\Sigma(D)$. %
The alert reader has spotted that the quasipositivity of $\Sigma(D)$ for $D$ of type~I also follows from \cref{thm:2}.
Nevertheless, we are going to supply an independent proof, which also serves as a warm-up for the other proofs.

The quasipositivity of links admitting diagrams of type~II can be shown using a method due to Baader \cite{Baa} as remarked by Tagami \cite{tagami}. It does not seem clear, however, how this approach could be strengthened to give a proof of strong quasipositivity.
For links admitting only diagrams of type~I, even quasipositivity has not hitherto been established.

The proof strategy is similar for both theorems. The quasipositivity of the surface $\Sigma(D)$ or $\Sigma'(D)$ is established by induction
over some measure of the complexity of the Seifert graph $\Gamma(D)$.
For the induction step, the following two facts about quasipositive Seifert surfaces are crucial:\medskip

\begin{enumerate}[label=(\Alph*)]
\setcounter{enumi}{12}\item
Murasugi sums of quasipositive Seifert surfaces are quasipositive \cite{qposplumbing},
\label{eq:murasugi}

\setcounter{enumi}{18}\item
Incompressible subsurfaces of quasipositive Seifert surfaces are quasipositive.
\label{eq:subsurface}%
\end{enumerate}%
Note that \cref{eq:subsurface} is an immediate consequence of the characterization of quasipositive Seifert surfaces that we use.
\subsection*{Outline of the paper}
The remainder of the paper contains the proofs of the main results.
The proof of \cref{thm:1} is split into \cref{prop:CanSurfType1} for type~I in \cref{sec:almost1},
and \cref{prop:mainII} for type~II in \cref{sec:almost2}.
\cref{thm:2} is proven in \cref{sec:canonical}.
\Cref{cor:thm2viagraph}, \Cref{prop:s=cang=>sqp} and \Cref{cor:g<=>cang=>sqp<=>s=g}
are proven in \cref{sec:corollaries}.
Sections \ref{sec:almost1}, \ref{sec:almost2}, \ref{sec:canonical} and \ref{sec:corollaries} can essentially be read independently.
\subsection*{Acknowledgments}
The authors thank Sebastian Baader both generally for his advocacy of quasipositivity and specifically for a conversation that inspired an important step in the proof of \cref{thm:2}.
They also thank an anonymous referee for a suggestion which lead to further applications of \cref{thm:2}.
The first author gratefully acknowledges support by the SNSF Grant 181199.
The second author is supported by 
the DFG, project no.~412851057.

\section{Almost positive links of type I}
\label{sec:almost1}
The goal of this section is to prove the following.
\begin{proposition}\label{prop:CanSurfType1}
The canonical surface $\Sigma(D)$ of a diagram $D$ of type~I is quasipositive.
\end{proposition}

Let us start by providing details regarding the Seifert graph $\Gamma(D)$ of a diagram~$D$, which was briefly defined in the introduction.

The set of edges adjacent to a vertex $k$ of $\Gamma(D)$ carries a cyclic ordering,
which comes from the ordering of crossings around the
Seifert circle $k$.  Moreover, $k$ separates $\mathbb{R}^2$ into an interior and an exterior.
So each edge adjacent to $k$ carries the additional information of \emph{on which side} of $k$ it lies.
We say two Seifert circles $k$ and $k'$ are \emph{nested} if one lies in the interior of the other.

If $D$ has no nested Seifert circles, that is the interior of every Seifert circle is empty, then shrinking every Seifert circle to a point
provides a canonical embedding of $\Gamma(D)$ into $\mathbb{R}^2$. Thus, if $D$ has no nested Seifert circles, we shall treat $\Gamma(D)$ as a plane graph (i.e.~a graph with a fixed embedding into $\mathbb{R}^2$).

\begin{figure}[hbt]
\begin{center}
(a) \includegraphics[width=.25\textwidth]{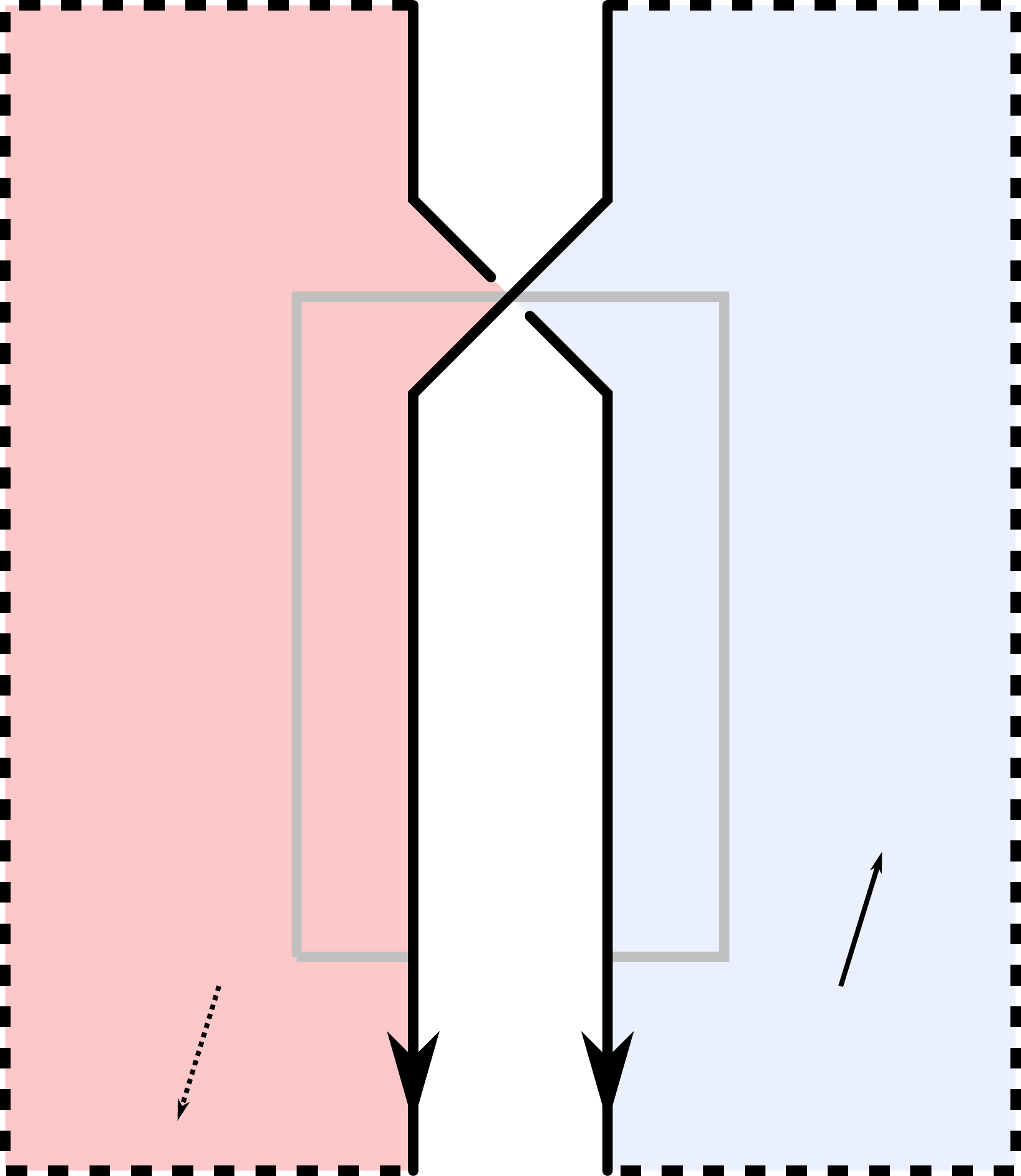}\hfill
(b) \includegraphics[width=.25\textwidth]{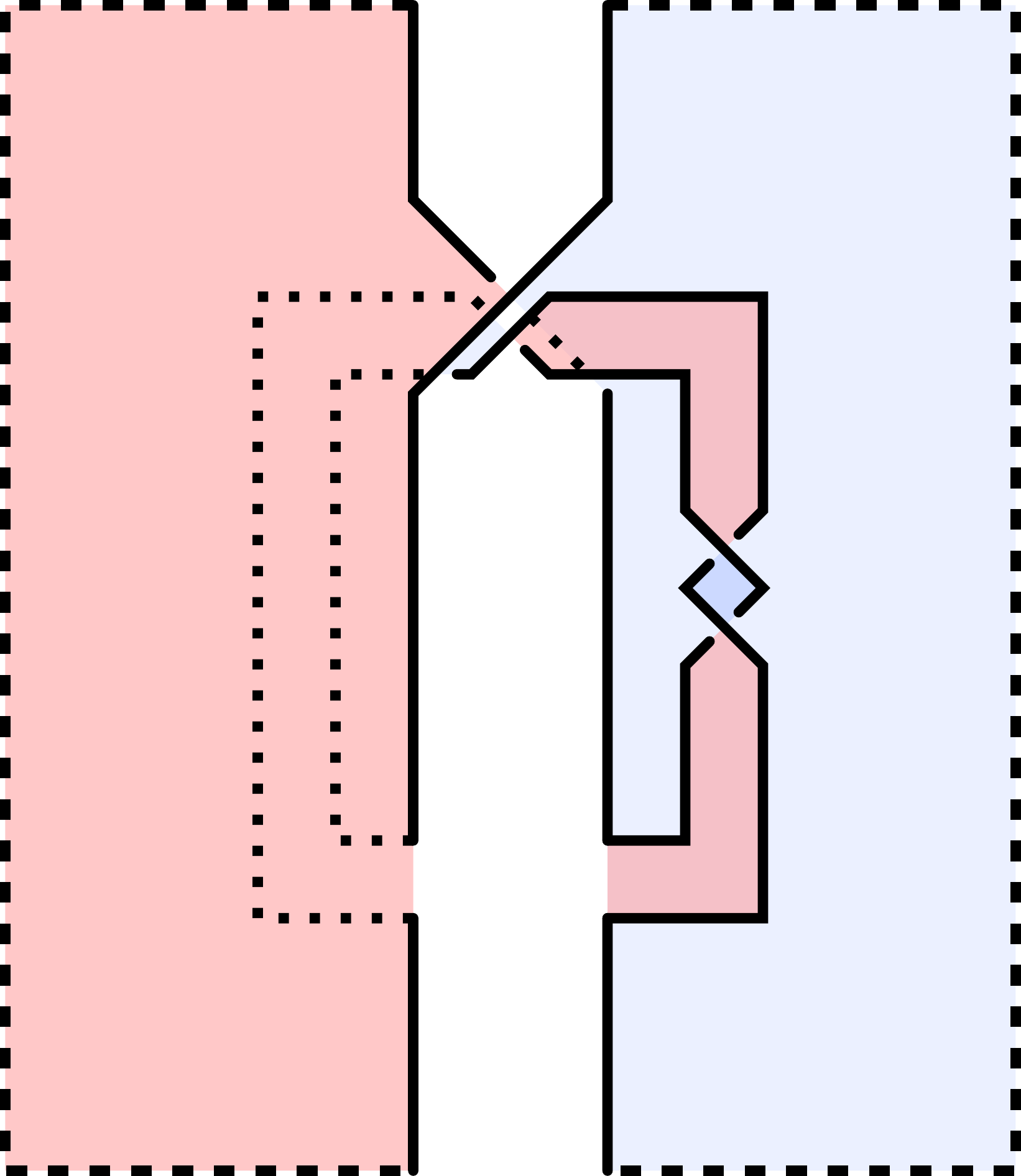}\hfill
(c) \includegraphics[width=.25\textwidth]{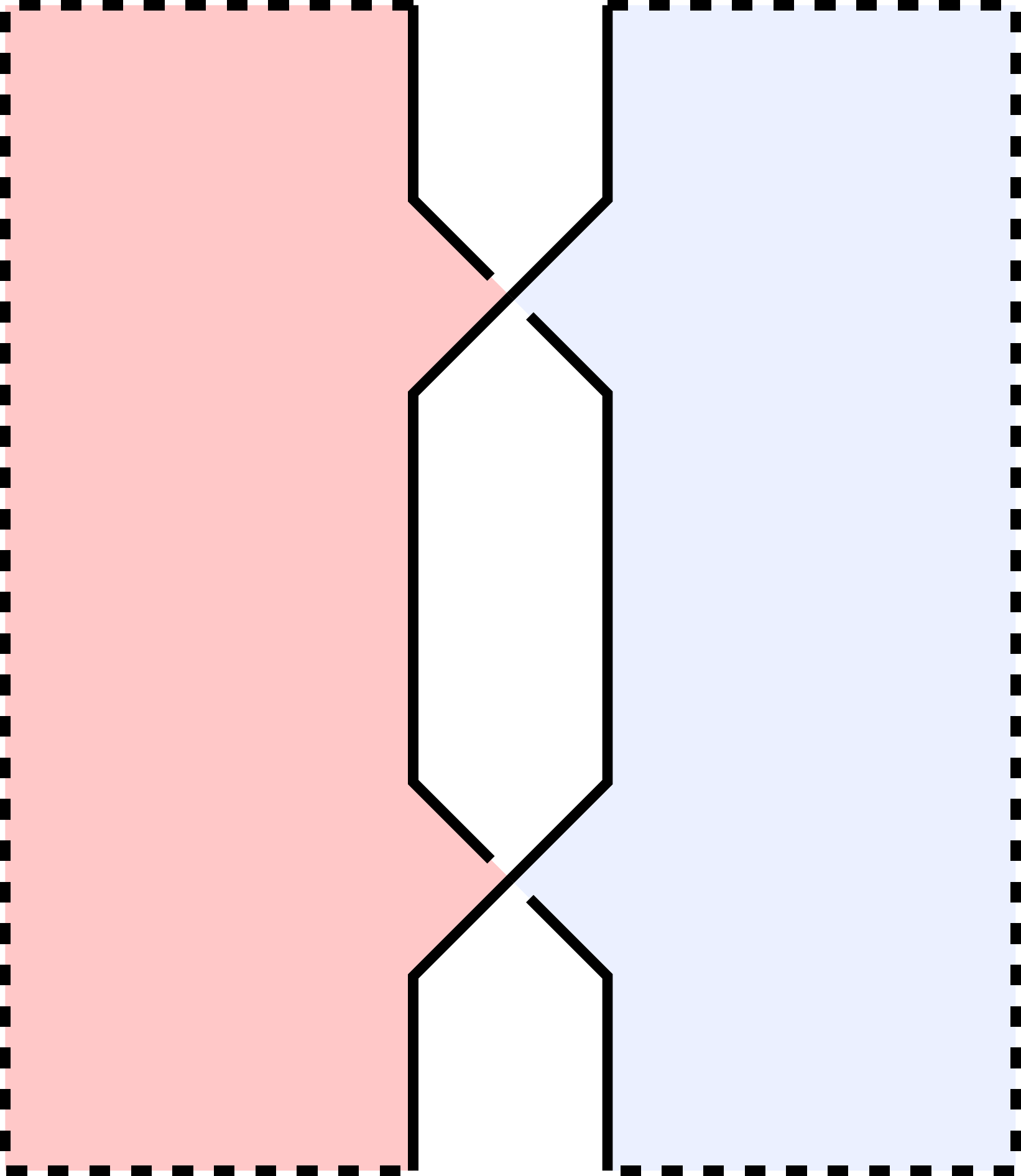}
\end{center}
\caption{Inserting a positive crossing next to another one by positive Hopf plumbing.
Red and blue indicate the two sides of oriented surfaces. Dotted lines are hidden below a surface.
\newline
\textbf{(a)} Two Seifert circles connected by a positive crossing. The small arrows indicate positive normal vectors of the surfaces.
\newline
\textbf{(b)} Surface obtained from (a) by plumbing a positive Hopf band along the gray curve on the positive side of the surface.
\newline
\textbf{(c)} This surface is isotopic to (b) (pull the Hopf band away from the crossing). Note that the central white region could contain infinity.
}
\label{fig:plumb}
\end{figure}

Next, we will need two lemmas giving sufficient diagrammatic conditions for the canonical surface being a Murasugi sum
and a Hopf plumbing, respectively.

\begin{lemma}[cf. \cite{cromwell-homogeneous}]
\label{lem:intext}
Let $D$ be a non-split link diagram (i.e.~$D$ is not a disjoint union of link diagrams) and let $k$ be a Seifert circle of~$D$.
Let $D_i$ and $D_e$ be the link diagrams forming the closure of the interior and the exterior of $k$, respectively (so that $D_i\cap D_e=k$). Then $\Sigma(D)$ is a Murasugi sum of $\Sigma(D_i)$ and $\Sigma(D_e)$.\qed\end{lemma}

\begin{lemma}\label{lem:plumb}
Let $D$ be a 
link diagram with a positive crossing $c$ between two Seifert circles $k$ and $k'$ that are not nested.
Let $D'$ be the diagram obtained from $D$ by inserting another positive crossing $c'$ that is parallel to $c$ and such that $c$ is the next crossing after $c'$ with respect to
both the cyclic orderings of crossings around $k$ and around $k'$.  Then $\Sigma(D')$ is the plumbing of $\Sigma(D)$ and a positive Hopf band.
\end{lemma}
\begin{proof}
See \cref{fig:plumb}.
\end{proof}

\begin{proof}[Proof of \cref{prop:CanSurfType1}]
We prove that $\Sigma(D)$ is quasipositive by induction over the sum of the number of Seifert circles of $D$ and the number of crossings of $D$.

Suppose that $D$ is a diagram with a Seifert circle that has empty exterior and non-empty interior.  By `moving infinity' we may change this to a diagram $D'$ with no such circle and such that $\Sigma(D') = \Sigma(D)$.  So without loss of generality we may assume that $D$ has no Seifert circle with empty exterior and non-empty interior.

Consider the following cases; we will prove below that they are exhaustive.
\begin{enumerate}
\item\label{c1}%
If $D$ consists of a single Seifert circle:
\newline
The canonical surface $\Sigma(D)$ is a disk, which is quasipositive.
\item\label{c2}%
If $D$ is split, i.e.\ $D = D_1 \sqcup D_2$ for link diagrams  $D_1$ and $D_2$:
\newline
The surfaces $\Sigma(D_i)$ are quasipositive by induction, and, thus, so is $\Sigma(D) = \Sigma(D_1) \sqcup \Sigma(D_2)$.
\item\label{c3}%
If there is a nugatory crossing (i.e.\ an edge removing which would disconnect~$\Gamma(D)$):
\newline
Let
$D'$ denote the diagram obtained by untwisting.
Then the surfaces $\Sigma(D)$ and $\Sigma(D')$ are isotopic and, by induction, $\Sigma(D')$ is quasipositive.
\item\label{c4}%
If $D$ has a Seifert circle with non-empty interior and non-empty exterior:
\newline
If $D$ is a split diagram, we proceed as in case (\cref{c2}). Otherwise, the canonical Seifert surface $\Sigma(D)$ is the Murasugi sum of two canonical surfaces by \cref{lem:intext}.
Since these two summands are quasipositive by induction, $\Sigma(D)$ is quasipositive by \cref{eq:murasugi}.
\end{enumerate}
\begin{enumerate}[resume]
\item\label{c5}%
If there is a Seifert circle with empty interior that is adjacent to exactly two crossings, one of which
is a positive crossing and the other is a negative crossing:
\newline
A Reidemeister-II-move removes that circle and the crossings adjacent to it, producing a diagram $D'$ such that $\Sigma(D)$ and $\Sigma(D')$ are isotopic and $\Sigma(D')$ is quasipositive by induction.
\item\label{c6}%
If a pair of non-nested Seifert circles are connected by two positive crossings that are next to each other (i.e.\ as in the hypothesis of \cref{lem:plumb}):
\newline
Denote by $D'$ the diagram obtained by deleting one of these crossings.
Then $\Sigma(D')$ is quasipositive by induction, and $\Sigma(D)$ is a plumbing of $\Sigma(D')$ and a positive Hopf band by \cref{lem:plumb}.  Thus $\Sigma(D)$ is quasipositive by \cref{eq:murasugi}.
\item\label{c7}%
If there is a closed interval embedded in the plane such that
\begin{enumerate}
\item its interior is disjoint from $D$,
\item its endpoints lie on two distinct Seifert circles $k_1, k_2$,
\item $k_1, k_2$ are oriented coherently,
\item \label{d4} if $k_1$ and $k_2$ are both connected to a third circle, then both of the connecting edges are positive:
\end{enumerate}
Denote by $D'$ the diagram obtained by adding a $1$--handle along that closed interval.
Because $D'$ has one fewer Seifert circle than $D$ and $D'$ is of type I by (\cref{d4}), $\Sigma(D')$ is quasipositive by induction.
Since $\Sigma(D')$ contains $\Sigma(D)$ as an incompressible subsurface, $\Sigma(D)$ is quasipositive by \cref{eq:subsurface}.
See \Cref{fig:pretzel} for an example of this case.
\end{enumerate}
\begin{figure}[th]
\includegraphics[width=.7\textwidth]{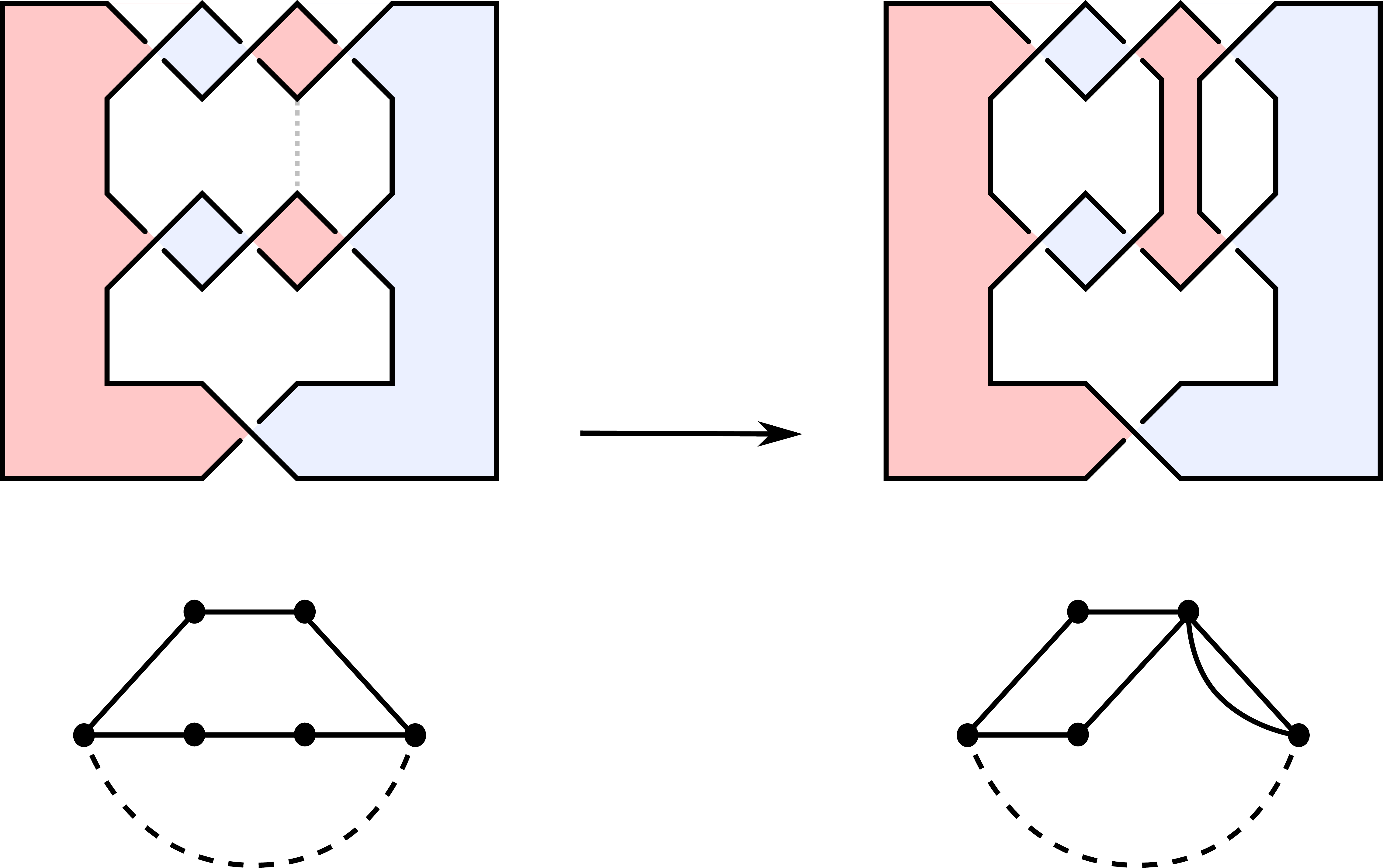}
\caption{Top and left: a type I diagram $D$ (the $(-3,-3,1)$--pretzel diagram) and $\Sigma(D)$. Below: its Seifert graph, with the unique edge of weight $-1$ drawn dashed.
On the right: a diagram $D'$ obtained from $D$ by applying (\cref{c7}) to the closed interval drawn gray and dotted, the surface $\Sigma(D')$, and the graph $\Gamma(D')$.}
\label{fig:pretzel}
\end{figure}
Let us prove that the above cases are exhaustive. For this, let us assume that (\cref{c1})--(\cref{c6}) are not satisfied,
and deduce that (\cref{c7}) is.
Note that the exclusion of (\cref{c2}) and (\cref{c4}) implies that no Seifert circles in $D$ are nested, so its Seifert graph $\Gamma(D)$ can be seen canonically as a plane graph. Furthermore, the exclusion of (\cref{c2}) implies that $\Gamma(D)$ is a connected graph, the exclusion of (\cref{c1}) and (\cref{c3}) imply that all vertices of $\Gamma(D)$ have degree at least $2$,
and by the exclusion of (\cref{c5}) vertices adjacent to a negative edge have degree at least $3$.

We distinguish two cases based on whether $\Gamma(D)$ contains a negative edge or not.
In both cases, we succeed in finding an interval as in (\cref{c7}). See \cref{fig:U1x} for an illustration.
\begin{figure}[hbt]
	\labellist
	\tiny
	\pinlabel {$k_1$}     at -100 640
	\pinlabel {$c$}       at  700 220
	\pinlabel {$c'$}      at 1350 940
	\pinlabel {$k$}       at 1150 -60
	\pinlabel {$k_2$}     at  800 1590
	\pinlabel {$k_2=k_1$} at 1950 440
	\pinlabel {$k_2=k_1$} at 2000 1290
	\pinlabel {$c$}       at 2750 210
	\pinlabel {$c'$}      at 2650 1210
	\pinlabel {$c''$}     at 2350 730
	\pinlabel {$k_3$}     at 2870 750
	\pinlabel {$k$}       at 3550 690
	\endlabellist
	\includegraphics[scale=0.075]{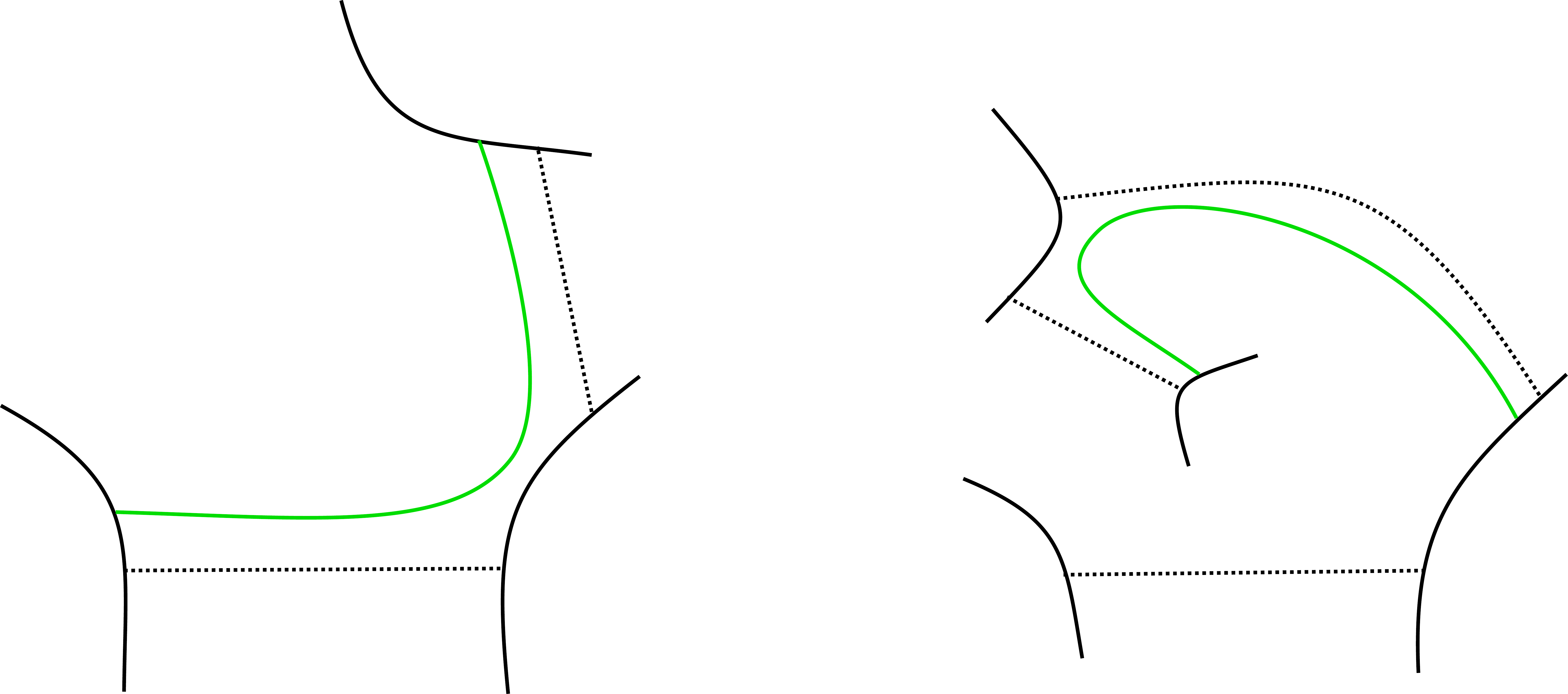}
	\caption{How to find an interval as in (\cref{c7}) (drawn green).}
	\label{fig:U1x}
\end{figure}
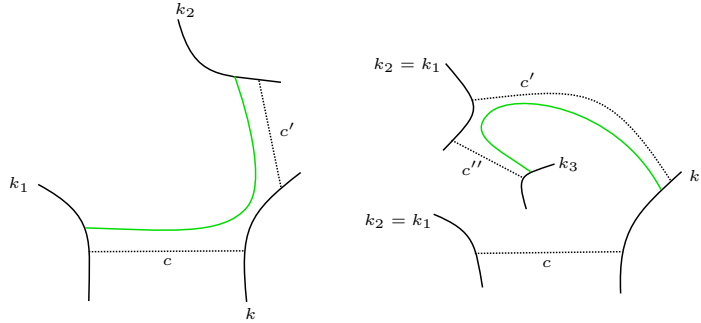

\textbf{Case 1.} Suppose there is no negative edge.
Pick any edge $c$ connecting circles $k$ and $k_1$. Since $k$ has degree at least 2, one may
walk from the edge $c$ in clockwise direction around $k$
until the next edge $c'$, which connects $k$ to some Seifert circle~$k_2$.

If $k_1 \neq k_2$ then there is a closed interval as in (\cref{c7}) which connects $k_1$ and $k_2$ by following $c$ and $c'$.  Note that because there is no negative edge, (\cref{d4}) is vacuously true.

If $k_1 = k_2$,
then walk clockwise around $k_2$ from where $c'$ meets $k_2$ until we meet the next edge $c''$.
If $c'' = c$, then we are in case (\cref{c6}).
If $c'' \neq c$, note that topologically $c''$ cannot connect to $k$ since we would then have met it when walking clockwise along $k$ from $c$ to $c'$.  So $c''$ connects to a Seifert circle $k_3 \neq k$.  There is then an interval connecting $k$ to $k_3$ by following $c'$ and $c''$.

\textbf{Case 2.} Now suppose there is a negative edge connecting circles $k$ and $k_0$.
Because $k$ has degree at least 3, one may walk from that edge in clockwise direction around $k$ until an edge $c$ connecting $k$
and some Seifert circle $k_1$, and still further until an edge $c'$ connecting $k$ and some Seifert circle $k_2$.
Because the diagram is type~I, one has $k_1 \neq k_0$ and $k_2 \neq k_0$.
Now one may proceed exactly as in the previous case to find an interval as needed for (\cref{c7}).
We note that the intervals as constructed above also satisfy (\cref{d4}) because none of $k_1, k_2$ and $k_3$ are adjacent to the negative edge.
\end{proof}

\section{Almost positive links of type II}
\label{sec:almost2}
In this section, we prove the second half of \cref{thm:1}.
\begin{proposition}\label{prop:mainII}
If $D$ is a link diagram %
of type II, then $D$ represents a strongly quasipositive link.
\end{proposition}
We first describe how to associate a Seifert surface $\Sigma'(D)$ to such a diagram $D$, which is similar to the canonical surface but with smaller first Betti number. Afterwards we will show that $\Sigma'(D)$ is a quasipositive Seifert surface.

\begin{construction}[Generalized Seifert algorithm]
Let $D$ be a diagram with exactly one negative crossing $c_-$. Further suppose $c_-$ is parallel to a positive crossing $c_+$. If there is more than one positive crossing parallel to $c_-$, we fix a choice of $c_+$.
We describe a version of Seifert's algorithm adapted to this setting that associates a Seifert surface $\Sigma'(D)$ %
with the diagram $D$ as follows.

Resolve all crossings except for $c_-$ and for $c_+$ in the oriented manner (as in Seifert's algorithm).
This produces a two-crossing diagram $D_0$ of an unlink $L_0$.
The diagram $D_0$ consists of two twice transversely intersecting curves, which we call $s_1$ and $s_2$, and simple closed curves that are pairwise disjoint and disjoint from $s_1$ and $s_2$, which we refer to as \emph{Seifert circles}. We take $s_1$ to be the curve that goes over.
We refer to the union of the Seifert circles and $\{s_1,s_2\}$ as \emph{generalized Seifert circles}.
For the rest of this section, we only consider $D$ such that $s_1$ is oriented clockwise and $s_2$ is oriented counterclockwise; which, if not the case, can be achieved by `moving infinity' without changing the associated link. See \cref{fig:s1s2}.

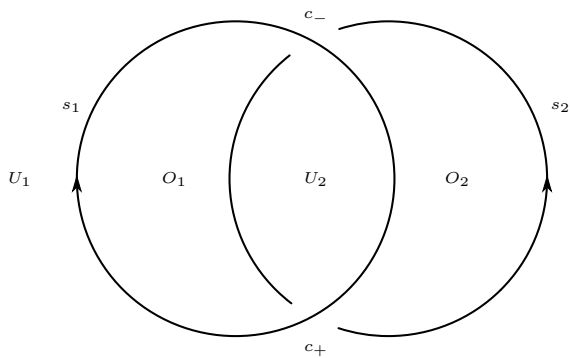
\begin{figure}[hb]
\centering
	\labellist
	\tiny
	\pinlabel {$s_1$} at    0  900
	\pinlabel {$s_2$} at 1900  900
	\pinlabel {$c_+$} at  950  -50
	\pinlabel {$c_-$} at  950 1250
	\pinlabel {$U_1$} at -200  620
	\pinlabel {$U_2$} at  950  620
	\pinlabel {$O_1$} at  400  620
	\pinlabel {$O_2$} at 1500  620
	\endlabellist
\includegraphics[width=.5\textwidth]{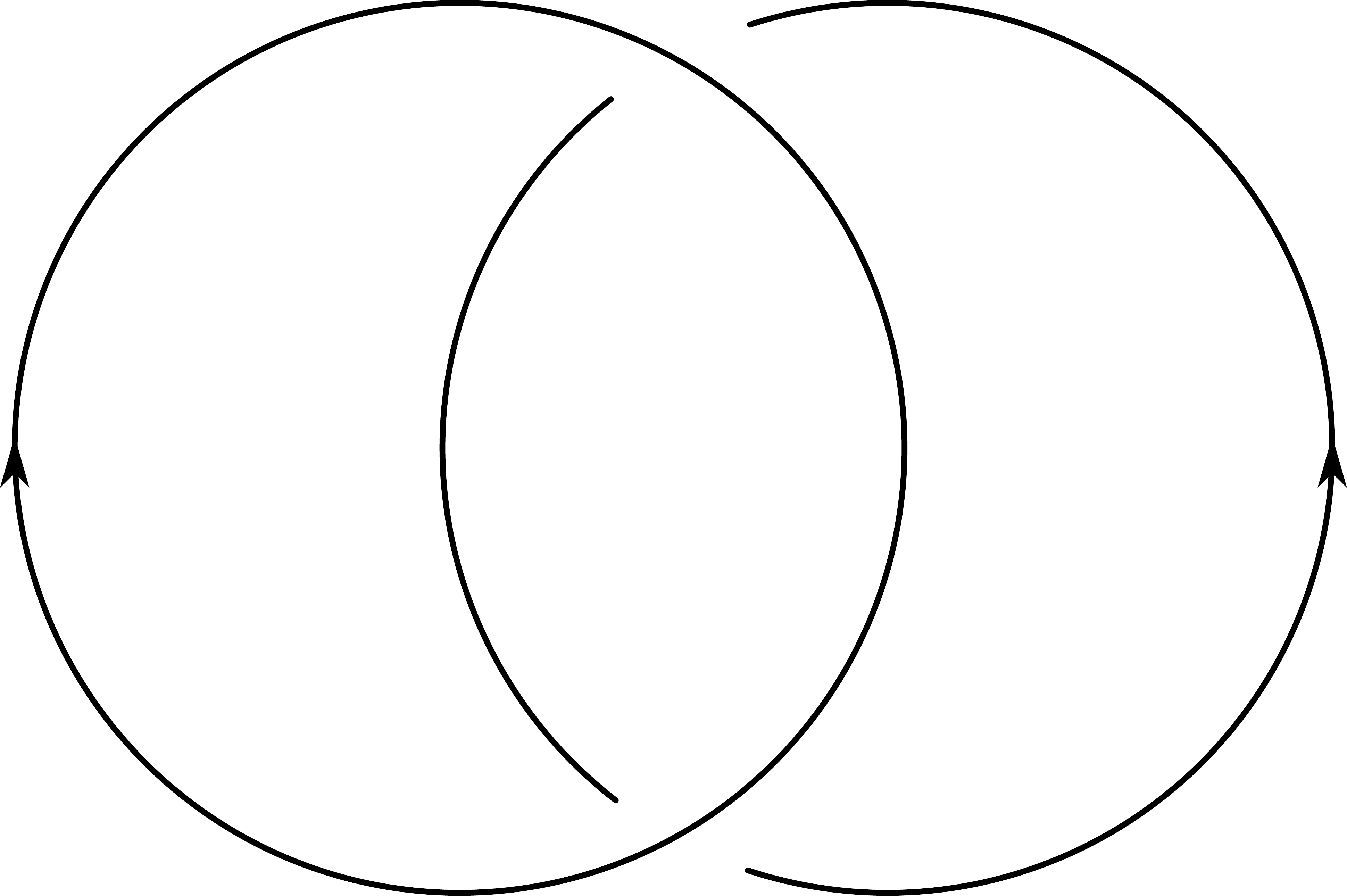}
\caption{$s_1$ and $s_2$ cut the plane into four regions.}
\label{fig:s1s2}
\end{figure}

As in Seifert's algorithm,
pick a disjoint union of oriented disks $d_i$ in $\R^3$ with constant $z$--coordinate, one for each generalized Seifert circle $k_i$, such that the boundary of $d_i$ projects to $k_i$ preserving orientation,
and glue in a twisted ribbon for each crossing to obtain $\Sigma'(D)$.
We choose the $z$--coordinates for the disks as follows.
\begin{enumerate}
  \item \label{item:1}
   The disk corresponding to $s_1$ has to lie above the disk corresponding to~$s_2$. In other words (using the convention that $s_1$ is oriented clockwise and goes over $s_2$), the positive sides of the disks face each other.
\item \label{item:2} Let $k_1$ be a generalized Seifert circle lying wholly inside a generalized Seifert circle $k_2$. The disk $d_1$ corresponding to $k_1$ lies to the positive side of the disk $d_2$ corresponding to $k_2$. In other words, a positive normal to $d_2$ points in the direction of $d_1$.
\end{enumerate}
Any such choice of $z$--coordinates assures that glueing in the twisted ribbons provides an embedded surface.

The choice of $z$--coordinate for disks corresponding to Seifert circles nested in $s_1$ and $s_2$ is crucial. For other disks, other choices work equally well in what is done below (save small changes in the details of the proof of \cref{lem:nextoisplumbing}).

 \end{construction}
\begin{remark}\label{rem:gSeifertAlg}
The above generalized Seifert algorithm produces a Seifert surface out of any link diagram $D$ with two marked crossings that are parallel and of opposite sign.

When no $c_-$ and $c_+$ are specified, the above algorithm produces a Seifert surface from a link diagram (simply ignore \eqref{item:1}).
However, that Seifert surface is in general not isotopic to the canonical surface, since
$z$--coordinates are usually chosen differently in the usual Seifert algorithm (nested implies higher, independent of the orientations of the disks). In this section, we write $\Sigma(\,\cdot\,)$ for the Seifert surface constructed by the Seifert algorithm using the height order given in \eqref{item:2}.

We note that the proof of \cref{prop:CanSurfType1} holds \emph{verbatim} for the $z$-coordinate conventions in this section.  Hence if $D$ is a positive diagram we already know that $\Sigma(D)$ is a quasipositive Seifert surface.
\end{remark}

In the usual way, we specify each crossing (different from $c_-$ and $c_+$) by giving an embedded closed interval $c$ (which we shall also call a \emph{crossing}) in $\mathbb{R}^2$.  The boundary points of each such crossing will lie on two different generalized Seifert circles, and each such crossing will be disjoint from $c_-$ and $c_+$, with its interior disjoint from all generalized Seifert circles. See \cref{fig:SpecialCases} for examples.

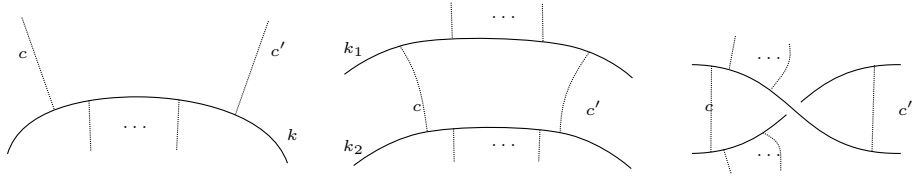
\begin{figure}[hb]
	\labellist
	\tiny
	\pinlabel {$\cdots$} at 3720 1180
	\pinlabel {$\cdots$} at 970 350
	\pinlabel {$\cdots$} at 3720 220
	\pinlabel {$\cdots$} at 5700 140
	\pinlabel {$\cdots$} at 5700 870
	\pinlabel {$c$} at 120 900
	\pinlabel {$c'$} at 2030 940
	\pinlabel {$k$} at 2130 300
	\pinlabel {$k_1$} at 2590 940
	\pinlabel {$k_2$} at 2590 200
	\pinlabel {$c$} at 3060 500
	\pinlabel {$c'$} at 4380 500
	\pinlabel {$c$} at 5240 500
	\pinlabel {$c'$} at 6710 500
	\endlabellist
	\includegraphics[scale=0.05]{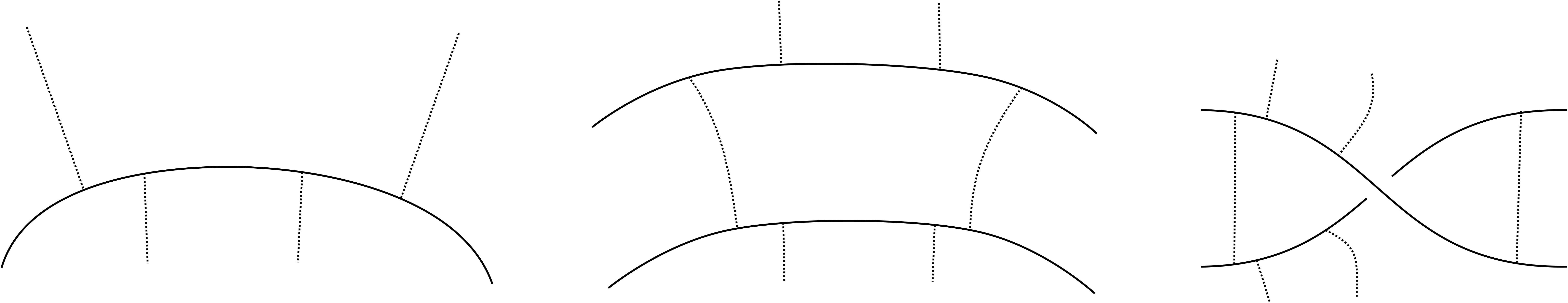}
	\caption{Left: Crossings $c$ and $c'$ next to each other on $k$.
		Middle and Right: Crossings $c$ and $c'$ next to each other.}
\label{fig:nextto}
\end{figure}
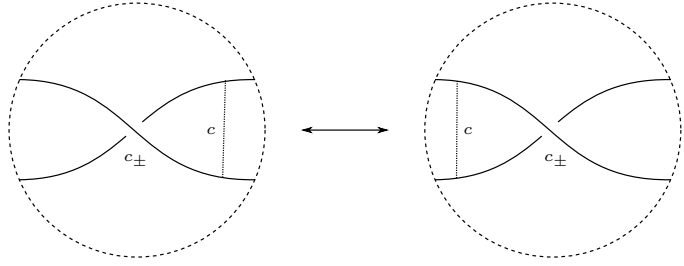
\begin{figure}[ht]
\begin{center}
	\labellist
	\tiny
	\pinlabel {$c_\pm$} at 850 650
	\pinlabel {$c_\pm$} at 3650 650
	\pinlabel {$c$} at 1350 850
	\pinlabel {$c$} at 3060 850
	\endlabellist
\includegraphics[width=.7\textwidth]{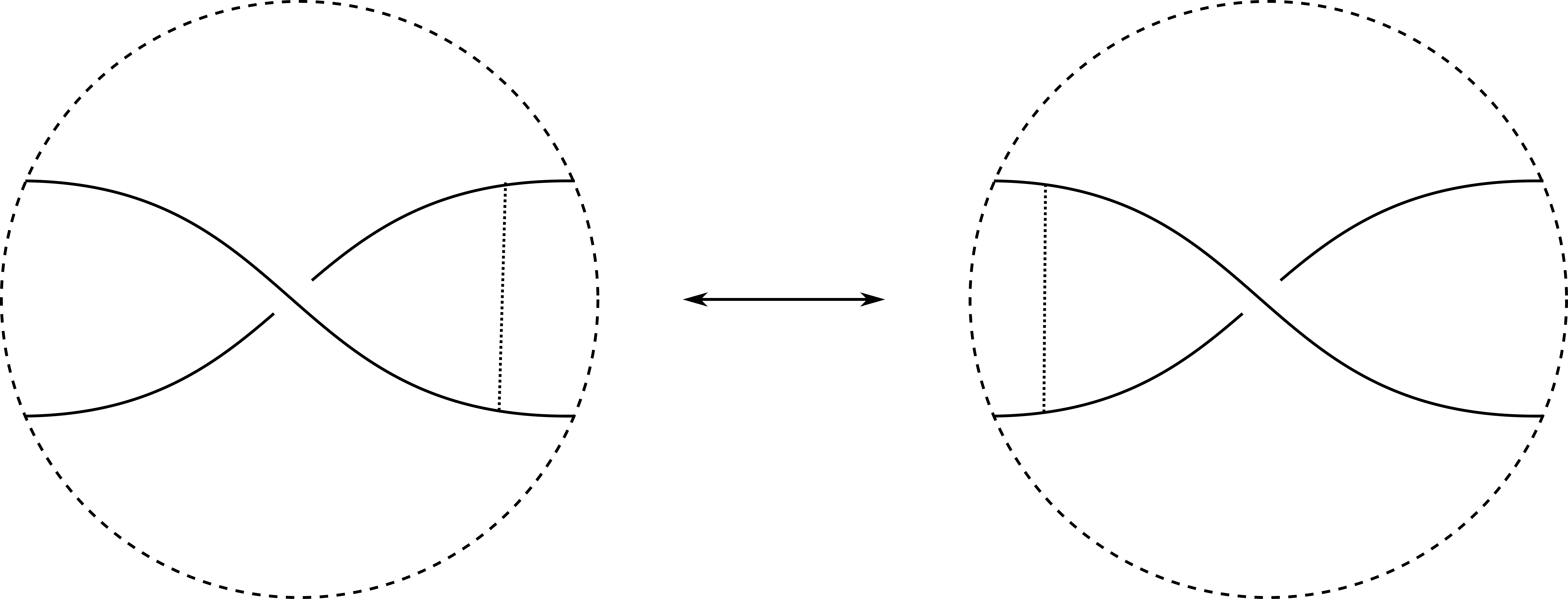}
\end{center}
\caption{
Swapping the crossing $c$, which is adjacent to $s_1$ and~$s_2$. Note that no other Seifert circles or crossings are present in the disk where the modification occurs.}
\label{fig:swapping}
\end{figure}

\begin{definition}
	We give a few notions which we shall refer to throughout the remainder of this section.
	
\begin{itemize}
\item From here on, unless otherwise stated, a \emph{crossing} refers to a crossing of $D$ that is neither $c_+$ nor $c_-$.
\item A crossing is said to be \emph{adjacent} to the generalized Seifert circles on which its endpoints lie.

\item Two crossings $c$ and $c'$ are said to be \emph{next to each other on a generalized Seifert circle} $k$, if they are both adjacent to $k$, they both lie to the same side of $k$, and there is a closed subinterval $I$ of $k$ with endpoints on $c$ and $c'$ such that $I$ does not contain $c_-$ or $c_+$ and there are no crossings with endpoints on $I$ that lie to the same side of $k$ as $c$ and $c'$. See \cref{fig:nextto}(left).

\item Two crossings $c$ and $c'$ are said to be \emph{next to each other}, if they are both adjacent to the same two generalized Seifert circles $k_1$ and $k_2$, they are next to each other on both $k_i$, witnessed by intervals $I_i$, such that the union $S=c \cup c' \cup I_1\cup I_2$ has the property that one of the two components of $\R^2\setminus S$ contains no generalized Seifert circles and no crossings. See \cref{fig:nextto}(middle).

\item We also say $c$ and $c'$ are \emph{next to each other} if they are next to each other after swapping one of them over $c_-$ or $c_+$; see \cref{fig:nextto}(right).

\item Here \emph{swapping a crossing $c$ over $c_-$} or \emph{over $c_+$} is the operation on diagrams defined by a modification of a diagram in a disk as described in \cref{fig:swapping}.
\item The union of $s_1$ and $s_2$ separates $\R^2$ into four regions.  Two of these have inconsistently oriented boundaries induced from the orientations of $s_1$ and~$s_2$.  We denote the unbounded region $U_1$ and the other $U_2$.  We denote the remaining two regions by $O_1$ and $O_2$, where $O_i$ is the region contained inside~$s_i$. See \cref{fig:s1s2}.
\end{itemize}
\end{definition}

We note that if the diagram $D'$ arises from $D$ by swapping a crossing, then $\Sigma'(D')$ and $\Sigma'(D)$ are isotopic Seifert surfaces.

\begin{proof}[Proof of \cref{prop:mainII}]

As in the proof of \cref{prop:CanSurfType1}, we shall proceed by induction on the sum of the number of Seifert circles and the number of crossings and consider a list of cases.  That these cases are exhaustive is the content of \cref{lem:cases_are_exhaustive} below.  We shall refer back to the proof of \cref{prop:CanSurfType1} for how to proceed with some of these cases.

\begin{figure}[ht]
	\label{fig:ocelot}
	\labellist
	\tiny
	\pinlabel {$D_1$} at -120 5300
	\pinlabel {$D_2$} at 4650 5300
	\pinlabel {$D_3$} at -120 3710
	\pinlabel {$D_4$} at 4650 3710
	\pinlabel {$D_5$} at -120 2120
	\pinlabel {$D_6$} at 4650 2120
	\pinlabel {$D_7$} at -120 530
	\pinlabel {$D_8$} at 4650 530
	\endlabellist
	\includegraphics[scale=0.06]{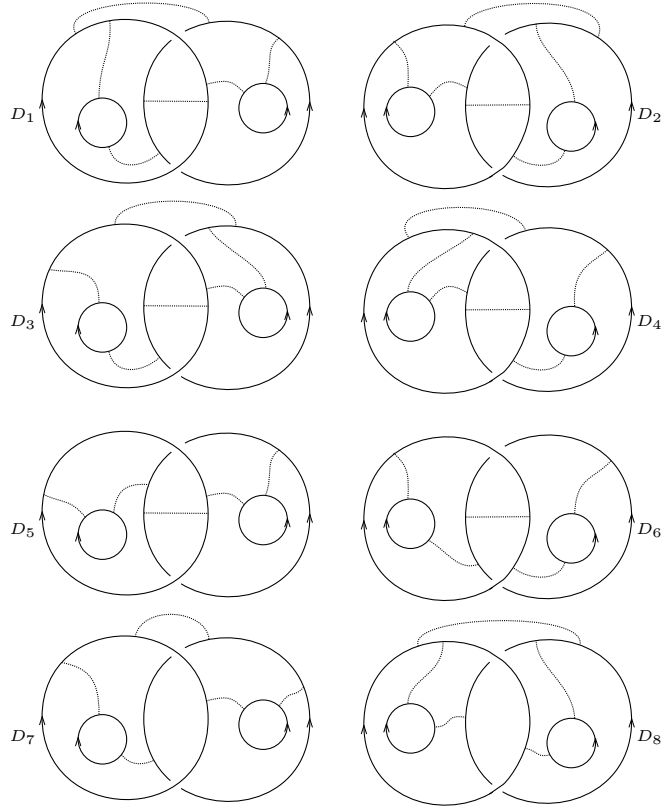}
	\caption{Diagrams to which \cref{c'split}--\cref{c'Interval} do not apply.}
\label{fig:SpecialCases}
\end{figure}

\begin{enumerate}[label=(\arabic*')]
\item\label{c'specalcase}
If $D$ is one of the 8 diagrams indicated in \cref{fig:SpecialCases} or a diagram obtained from one of them by deleting Seifert circles or crossings:
\newline
The Seifert surface $\Sigma'(D)$ is a quasipositive Seifert surface, as demonstrated in \cref{lem:special_cases_are_qp}.
\item\label{c'split} If $D$ is split:
\newline
Proceed as in (2) (using $\Sigma'(\,\cdot\,)$ rather than $\Sigma(\,\cdot\,)$ for the part of the diagram that contains the $s_i$). Explicitly, writing the diagram $D$ as $D_- \sqcup D_+$, where $D_-$ contains $c_-$ (and thus also $c_+$), we have that $\Sigma'(D) = \Sigma'(D_-) \sqcup \Sigma(D_+)$
is quasipositive, since $\Sigma'(D_-)$ is quasipositive by induction, and $\Sigma(D_+)$ is quasipositive because $D_+$ is positive.
\item If there is a nugatory crossing:
\newline
Proceed as in (3). %
\item\label{c'murasugi} If a Seifert circle has non-empty interior and exterior:
\newline
Proceed as in (4).
\item\label{c'nextto} If two generalized Seifert circles are connected by two crossings that are next to each other:
     \newline
     By \cref{lem:nextoisplumbing} (analog of \cref{lem:plumb} provided at the end of this section), we obtain a diagram $D'$ by removing one of the two crossings such that $\Sigma'(D)$ is quasipositive if and only if $\Sigma'(D')$ is quasipositive.  However, by induction, $\Sigma'(D')$ is a quasipositive Seifert surface.
\item\label{c'Interval} If there is a closed interval embedded in $\R^2$ such that
\begin{enumerate}
\item its interior is disjoint from $D$,
\item one of its endpoints lies on a Seifert circle $k_1$ and the other endpoint lies on a generalized Seifert circle $k_2$, and
\item $k_1, k_2$ are oriented coherently:
\end{enumerate}
Denote by $D'$ the diagram obtained by adding a $1$--handle along that closed interval.
Because $D'$ has one fewer Seifert circle than $D$, $\Sigma'(D')$ is quasipositive by induction.
And since $\Sigma'(D')$ contains $\Sigma'(D)$ as an incompressible subsurface, $\Sigma'(D)$ is quasipositive by \cref{eq:subsurface}.\qedhere
\end{enumerate}
\end{proof}
Let us first establish that the above cases are exhaustive.

\begin{lemma}\label{claim:regions}
If the conditions of none of \cref{c'split}--\cref{c'Interval} are satisfied, then

\begin{enumerate}[label=\roman*)]
\item the regions $U_1$ and $U_2$ contain no Seifert circles,

\item $O_1$ and $O_2$ each contain at most one Seifert circle,

\item each of the Seifert circles has exactly 2 positive crossings adjacent to it,
\item in each $U_i$ there is at most one crossing between $s_1$ and $s_2$.
\end{enumerate}
\end{lemma}

We postpone the proof of \cref{claim:regions} and apply it to %
prove the following.
\begin{lemma}
	\label{lem:cases_are_exhaustive}
If $D$ is a diagram such that \cref{c'split}--\cref{c'Interval} do not apply, then \cref{c'specalcase} applies to $D$.
\end{lemma}
\begin{proof}
By \cref{claim:regions}, it suffices to consider diagrams satisfying i)--iv).

First we consider the case where $D$ has a crossing in both $U_1$ and $U_2$ and $O_1$ and $O_2$ each contain a Seifert circle.  Once one has fixed the endpoints of the crossings in $U_1$ and $U_2$, there are four possibilities for how the two crossings adjacent to the unique Seifert circle in $O_1$ can lie without being next to each other (as otherwise~\cref{c'nextto} applies to $D$).
Similarly, there are four possibilities for how the two crossings adjacent to the Seifert circle in $O_2$ can lie without being next to each other.

Thus in this case there are 16 diagrams satisfying i)--iv). However, in 12 of these diagrams there is a crossings in $U_1$ that is next to a crossing of $U_2$ (using the notion of next to each other that uses swapping).  Hence for these 12 diagrams \cref{c'nextto} applies.  The four remaining diagrams, which we denote by $D_1$, $D_2$, $D_3$, and $D_4$, are indicated in \cref{fig:SpecialCases}.

Next we consider the case where $D$ has a crossing in exactly one of $U_1$ or $U_2$, and two Seifert circles.  There are eight such diagrams satisfying i)--iv). Four of these arise by deleting a crossing in one of the diagrams $D_1$, $D_2$, $D_3$, and $D_4$. The other four, denoted by $D_5$, $D_6$, $D_7$, and $D_8$, are indicated in \cref{fig:SpecialCases}.

Finally, it is easy to see that any case not yet considered is obtained from at least one of the $D_i$ by deleting Seifert circles or crossings.
\end{proof}

\begin{proof}[Proof of \cref{claim:regions}] $\,$

\begin{enumerate}[label=\textbf{\roman*)},leftmargin=1.9em]
\item
Assume towards a contradiction that there is at least one Seifert circle, say $k_1$, in $U_i$.  There is a crossing $c$ connecting $k_1$ to a different generalized Seifert circle~$k$. We distinguish two cases.

\textbf{Case 1:} Assume that on $k$ there is a crossing $c'$ %
next to $c$.
Let $k_2$ be the generalized Seifert circle adjacent to $c'$ that is not $k$; see \cref{fig:U1x}.

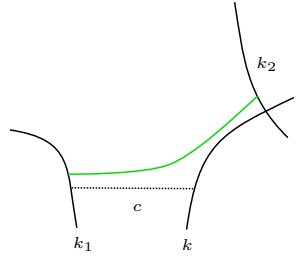
\begin{figure}[hbt]
	\labellist
	\tiny
	\pinlabel {$k$}       at 960 -80
	\pinlabel {$k_1$}     at 400 -80
	\pinlabel {$k_2$}     at 1390 900
	\pinlabel {$c$}       at 690 120
	\endlabellist
	\includegraphics[scale=0.075]{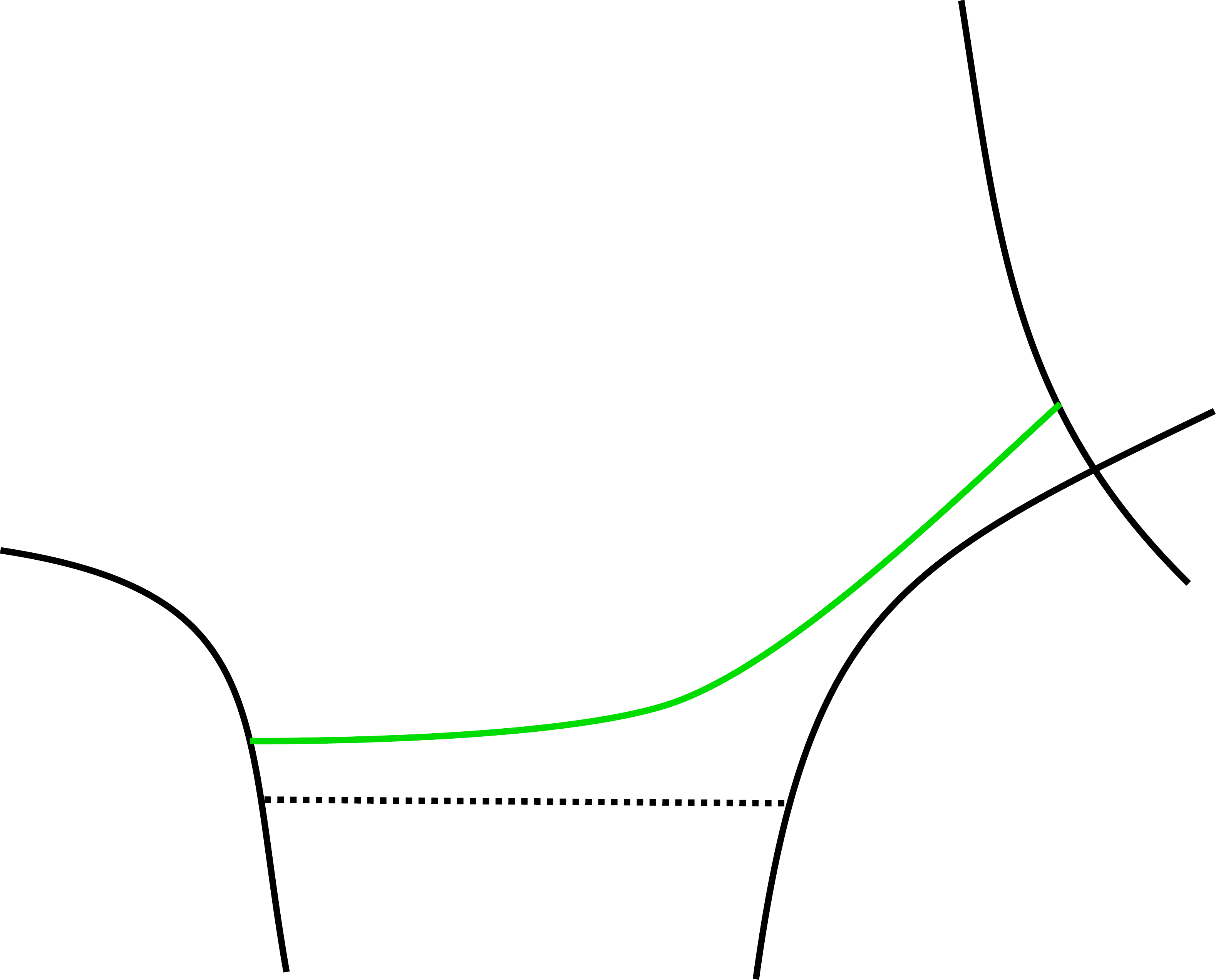}
	\caption{
$k$ is $s_1$ or $s_2$ and $c$ is next to one of the two intersection points of $s_1$ and $s_2$ (bottom).}
	\label{fig:U1y}
\end{figure}

If $k_1\neq k_2$, then \cref{c'Interval} applies (see green interval in \cref{fig:U1x}(left)), hence we obtain a contradiction.
Thus we have that $k_1=k_2$. The crossings $c'$ and $c$ (compare \cref{fig:U1x}) are next to each other on $k$, which implies that there must be another crossing $c''$ adjacent to $k_2$ that is between $c$ and $c'$, since otherwise the crossings $c$ and $c'$ are next to each other (this uses that $D$ is not split). Let $k_3$ be the other Seifert circle adjacent to $c''$. Now \cref{c'Interval} applies (see green interval in \cref{fig:U1x} (right)), hence we obtain a contradiction.

\textbf{Case 2:} Assume that on $k$ there is no crossing next to $c$. This implies that $k$ is $s_1$ or $s_2$ (if $k$ were a Seifert circle, having no crossing next to $c$ on $k$ would imply that $c$ is the only crossing on one side of $k$, thus $c$ would be nugatory).

Thus, \cref{c'Interval} applies (see green interval in \cref{fig:U1y}), contradiction.

\item
Assume towards a contradiction that there are at least two Seifert circles in $O_1$ (without loss of generality).

\textbf{Case 1:} Assume that inside $O_1$ there are two distinct crossings $c$, $c'$ adjacent to $s_i$ for some $i$.  Calling the Seifert circle $k_1$ that is also adjacent to $c$, we can now argue verbatim as in Case 1 of i) above.

\textbf{Case 2:} Assume that inside $O_1$ there is at most one crossing adjacent to $s_1$ and at most one crossing adjacent to $s_2$.  Then, to avoid nugatory crossings, we know that there is exactly one crossing $c_1$ adjacent to $s_1$ and exactly one crossing $c_2$ adjacent to $s_2$.

Call the Seifert circle $k_1$ that is also adjacent to $c_1$.  If $k_1$ is adjacent to no other crossings then $c_1$ is nugatory.  If $k_1$ is adjacent to $c_1$, to $c_2$, and to no other crossing then either the diagram is disconnected, $k_1$ has non-empty interior, or there is no other Seifert circle in $O_1$.  If $k_1$ is adjacent to some crossing different from $c_1$ and $c_2$, then pick $c\not=c_2$ to be a crossing next to $c_1$ on $k_1$.  Then there is an arc connecting $s_1$ to the other Seifert circle adjacent to $c$.

\item
Let $i \in \{ 1 , 2 \}$.  Let $k$ be a Seifert circle in $O_i$. All crossings adjacent to $k$ are adjacent to $s_1$ or $s_2$ since there are no other Seifert circles in $O_i$ by ii).  If there are at least three crossings, then two of them are adjacent to the same $s_j$ and are next to each other since there are no other Seifert circles in $O_i$ by ii).

\item
Let $i \in \{ 1 , 2 \}$.  If there are two or more crossings in $U_i$, then two of them are next to each other since there are no Seifert circles in $U_i$ by i).\qedhere
\end{enumerate}%
\end{proof}

It remains to show that for the diagrams $D_i$ for $i \in \{ 1, \ldots, 8\}$ given in \cref{fig:SpecialCases}, $\Sigma'(D_i)$ is a quasipositive surface. This implies that $\Sigma'(D)$ is a quasipositive Seifert surface for any diagram $D$ obtained from some $D_i$ by deleting crossings or Seifert circles.   (Because in this case $\Sigma'(D)$ is an incompressible subsurface of $\Sigma'(D_i)$ and thus a quasipositive Seifert surface by~\cref{eq:subsurface}.)

\begin{lemma}
	\label{lem:special_cases_are_qp}
The Seifert surface $\Sigma'(D_i)$ is quasipositive for all $i \in \{ 1, \ldots , 8 \}$.
\end{lemma}
\begin{proof}%
We discuss each of the surfaces $\Sigma'(D_i)$ in turn.

\begin{itemize}[leftmargin=.9em]
\item $\mathbf{\Sigma'(D_1).}$ We write  $L_1$ for the boundary of $\Sigma'(D_1)$.  Note that $L_1$ has two components and further note that $\Sigma'(D_1)$ has Euler characteristic $\chi(\Sigma'(D_1)) = -2$.  Hence $\Sigma'(D_1)$ is a twice punctured surface of genus 1.

By inspection, $L_1$ is the two component link consisting of the positive trefoil and a meridian positively linking the trefoil.

Note that the link $L_1$ is the boundary of the surface $F$ given as the connected sum (a special case of a Murasugi sum) of the fiber surface of the positive trefoil $F_{2,3}$ and the positive Hopf band $F_{2,2}$.  In particular, $F$ is a fiber surface (since Murasugi sum preserves fiberedness) for $L$
with Euler characteristic $-2$.
Since $F$ is a fiber surface, it is the unique %
Euler characteristic maximizing Seifert surface for~$L_1$; therefore $\Sigma'(D_1)$ is isotopic to $F$.  However, $F$ is a quasipositive Seifert surface by \cref{eq:murasugi} since it is the Murasugi sum of the two quasipositive Seifert surfaces $F_{2,3}$ and $F_{2,2}$.

\item $\mathbf{\Sigma'(D_2).}$ This is seen to be isotopic to $\Sigma'(D_1)$, for example via %
rotation about the vertical axis in the plane.

\item $\mathbf{\Sigma'(D_3).}$ %
We consider a diagram move (similar to swapping a crossing) indicated in \cref{fig:swappingCircletocrossing} that swaps a Seifert circle with two adjacent crossings into a crossing and vice versa.
\begin{figure}[ht]
\begin{center}
	\labellist
	\tiny
	\pinlabel {$O_1$} at 2950 850
	\pinlabel {$O_2$} at 4200 850
	\pinlabel {$c_+$} at 3550 650
	\endlabellist
\includegraphics[width=.9\textwidth]{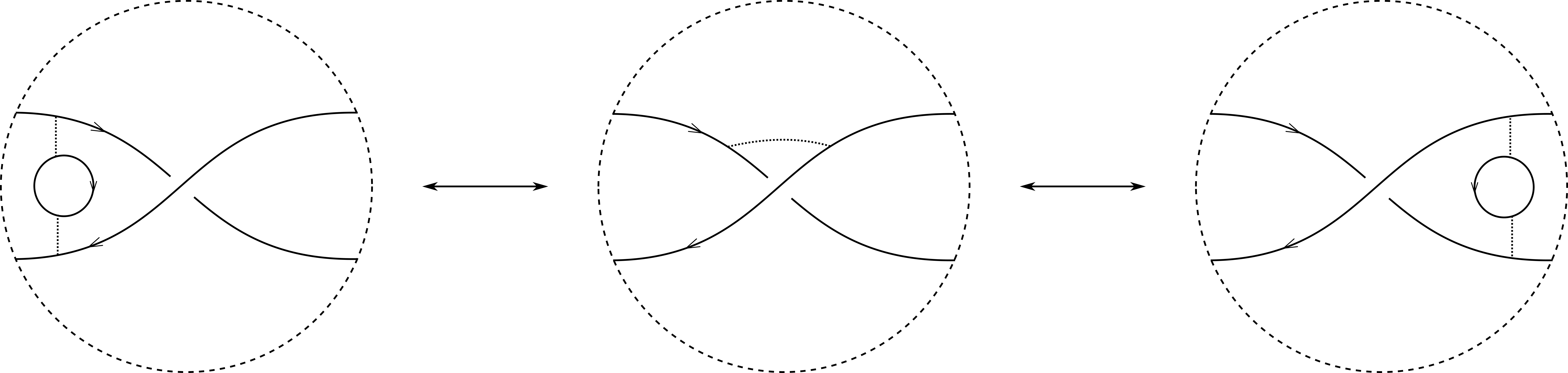}
\end{center}
\caption{Left-to-middle and right-to-middle: swapping a Seifert circle over $c_+$ into a crossing.
\newline
Left-to-right: swapping a Seifert circle from $O_1$ over $c_+$ into a Seifert circle in $O_2$.}
\label{fig:swappingCircletocrossing}
\end{figure}
Note that $\Sigma'(D)=\Sigma'(D')$ for diagrams $D$ and $D'$ that are related by this move.

We apply this move (\cref{fig:swappingCircletocrossing}(left-to-middle)) to $D_3$: swap the Seifert circle in $O_1$ over $c_+$ into a crossing in $U_2$ to get a diagram $D_3'$ with two crossings in $U_2$ and one crossing in $U_1$. One of the crossings in $U_2$ is next to the other crossing in $U_2$ and also next to the crossing in $U_1$.
By~\cref{lem:nextoisplumbing}, $\Sigma'(D_3')$ is quasipositive if and only if $\Sigma'(D_3'')$ is quasipositive, where $D_3''$ is the diagram obtained from $D_3'$ by deleting the crossing in $U_1$ and one of the crossings in $U_2$.

Finally, we observe that $\Sigma'(D_3'')$ is a quasipositive surface. This follows since $D_3''$ can be obtained from $D_1$ by deleting crossings and Seifert circles establishing that $\Sigma'(D_3'')$ is an incompressible subsurface of the quasipositive surface $\Sigma'(D_1)$.

\item $\mathbf{\Sigma'(D_4).}$
This is seen to be isotopic to $\Sigma'(D_3)$, for example via %
rotation about the vertical axis in the plane.

\item $\mathbf{\Sigma'(D_5).}$ %
First use the move depicted in \cref{fig:swappingCircletocrossing}(middle-to-left) to swap the crossing in $U_2$ over $c_+$ to result in a diagram with two Seifert circles in $O_1$.  Then swap the Seifert circle in $O_2$ over $c_+$ to $O_1$ using the move depicted in \cref{fig:swappingCircletocrossing}(right-to-left) resulting in a diagram $D_5'$ with three Seifert circles in $O_1$.

The Seifert surface $\Sigma'(D_5')$ is isotopic to $\Sigma'(D_5)$ and is easily seen to be the Seifert surface of a positive diagram of the $(-2,-2,-2)$--pretzel link, and hence quasipositive.

\item $\mathbf{\Sigma'(D_6).}$ The surface $\Sigma'(D_6)$ is a thrice-punctured sphere since it has Euler characteristic $-1$ and its boundary is a three component link $L_6$.

By inspection, the link $L_6$ is the three-component link given as an unknot with two parallel positively linked meridians. Thus, we note that $L_6$ is the boundary of the Seifert surface $S$ of Euler characteristic $-1$ given as the connected sum of two positive Hopf bands. As for $\Sigma'(D_1)$, we conclude that $S$ is a quasipositive Seifert surface that is isotopic to $\Sigma'(D_6)$. Thus, $\Sigma'(D_6)$ is a quasipositive Seifert surface.%

\item $\mathbf{\Sigma'(D_7)}$. The surfaces $\Sigma'(D_7)$ and $\Sigma'(D_6)$ are isotopic since $D_6$ can be turned into $D_7$ by swapping a crossing.

\item $\mathbf{\Sigma'(D_8)}$. The surfaces $\Sigma'(D_8)$ and $\Sigma'(D_5)$ are isotopic since swapping a crossing in $D_5$ (and then moving the crossing over infinity) turns $D_5$ into~$D_8$.\qedhere
\end{itemize}
\end{proof}

We end the section with the generalization of \cref{lem:plumb} used above.

\begin{lemma}\label{lem:nextoisplumbing}
Let $D$ be a diagram with two marked crossings of opposite sign $c_-$ and $c_+$.
If two positive crossings $c_1$ and $c_2$ are next to each other, then for some $i \in \{1,2\}$, the diagram $D'$ obtained by deleting $c_i$ satisfies:
$\Sigma'(D)$ is a quasipositive Seifert surface if and only $\Sigma'(D')$ is a quasipositive Seifert surface.
\end{lemma}
\begin{proof}[Proof of \cref{lem:nextoisplumbing}]
One direction of the Lemma follows immediately since $\Sigma'(D')$ is an incompressible subsurface of $\Sigma'(D)$.
The other direction shall be proven in a similar way as \cref{lem:plumb}.
We may and do assume that $c_1$ and $c_2$ are next to each other without swapping needed (otherwise swap a crossing first and consider the resulting diagram as $D$).
Writing $k$ and $k'$ for the generalized Seifert circles adjacent to $c_1$ and $c_2$, we wish to prove that the local situation
is isotopic to \cref{fig:nexttoplumbing}(a). Then, while there may be twisted ribbons attached to $k$ and to $k'$ between $c_1$ and $c_2$,
either all of those ribbons lie above the disk corresponding to $k$, and below the disk corresponding to $k'$, or vice versa.
So, the Hopf band in \cref{fig:nexttoplumbing}(b) may be plumbed to \cref{fig:nexttoplumbing}(c) such that it does not interfere with the ribbons.

To make this plan work, we distinguish two cases depending on whether $k$ and $k'$ are nested or not;
in each of the cases we pay attention to the possibility that $k$ and $k'$ may be $s_1$ or $s_2$.

\textbf{Case when $k$ and $k'$ are not nested.} It turns out that $D'$ can be chosen such that $\Sigma'(D)$ arises from plumbing a positive Hopf band to $\Sigma'(D')$, which implies that $\Sigma'(D)$ is a quasipositive Seifert surface if and only if $\Sigma'(D')$ is. The argument is more involved version of the proof of \cref{lem:plumb}; in particular, we will have to be careful which of the two crossings $c_1$ and $c_2$ to eliminate in $D$ to obtain $D'$.

We first consider the case that at least one of the generalized Seifert circles $k$ or $k'$  is a Seifert circle (i.e.~not an $s_i$). The situation is as depicted in \cref{fig:nexttoplumbing}(a) (we note that picture only depicts a closed range of $z$--coordinate height---far above or below there could be further disks corresponding to generalized Seifert circles that contain both $k$ and $k'$).
\begin{figure}[ht]
\begin{center}
(a) \includegraphics[width=.27\textwidth]{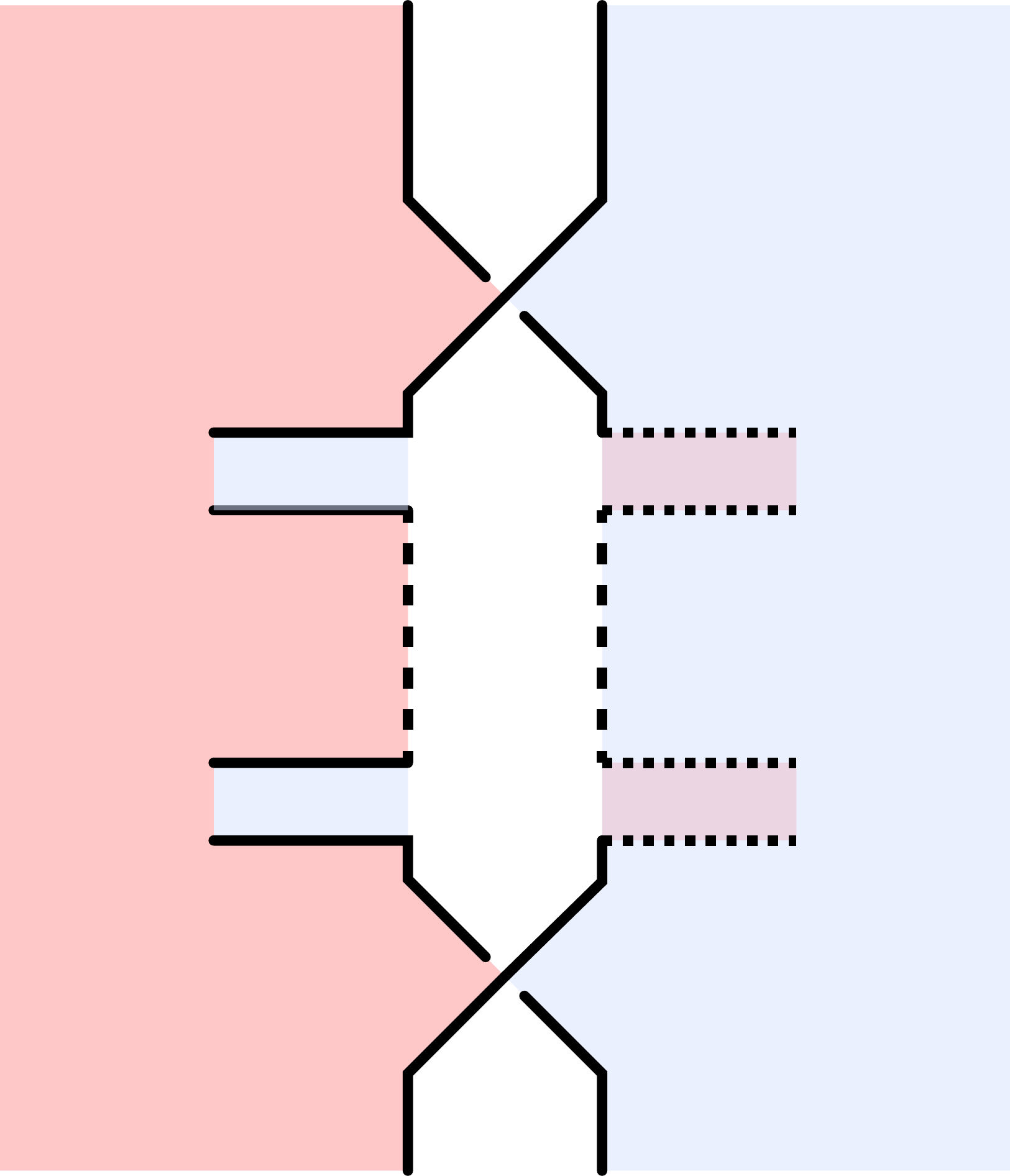}\hfill
(b) \includegraphics[width=.27\textwidth]{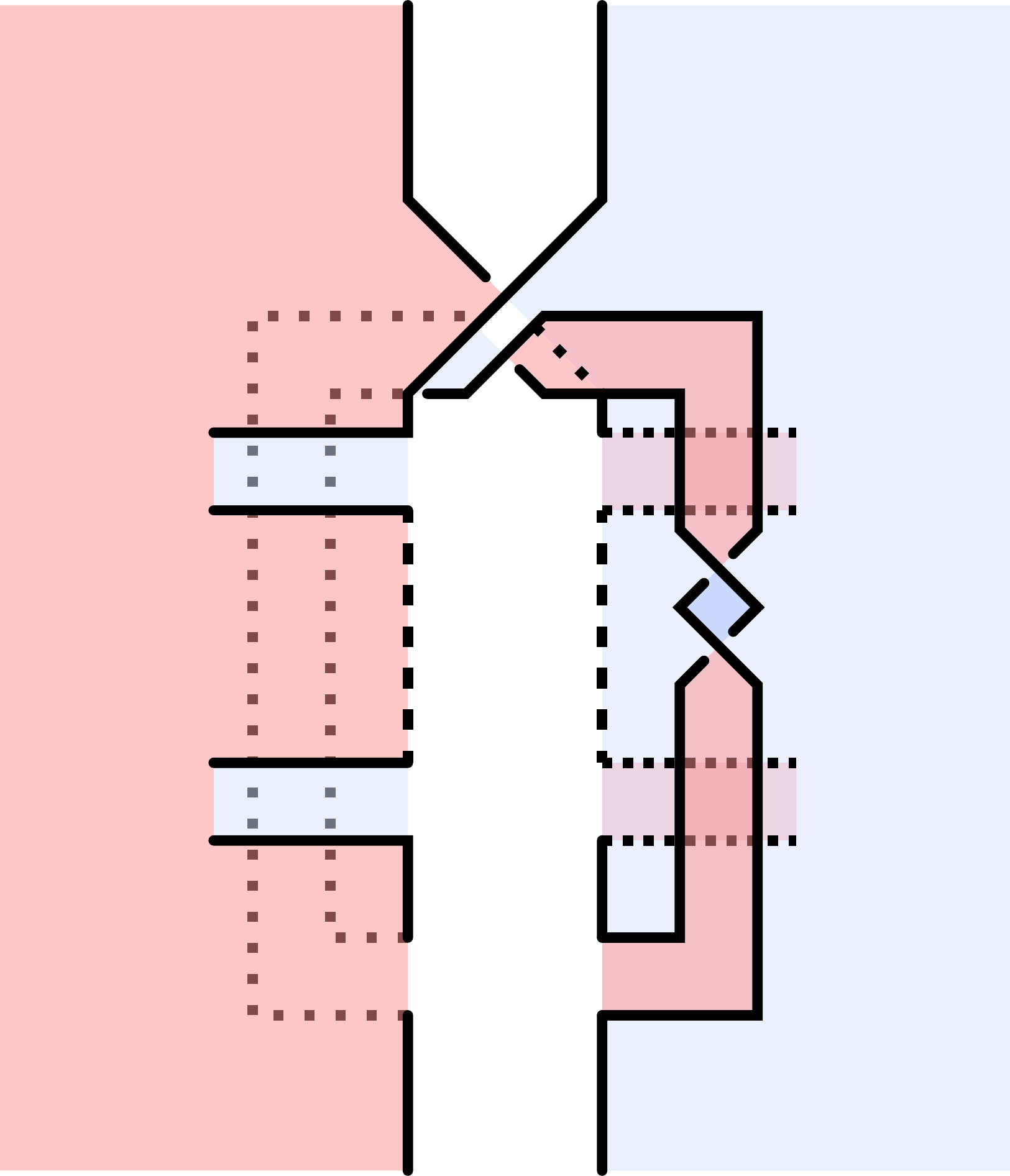}\hfill
(c) \includegraphics[width=.27\textwidth]{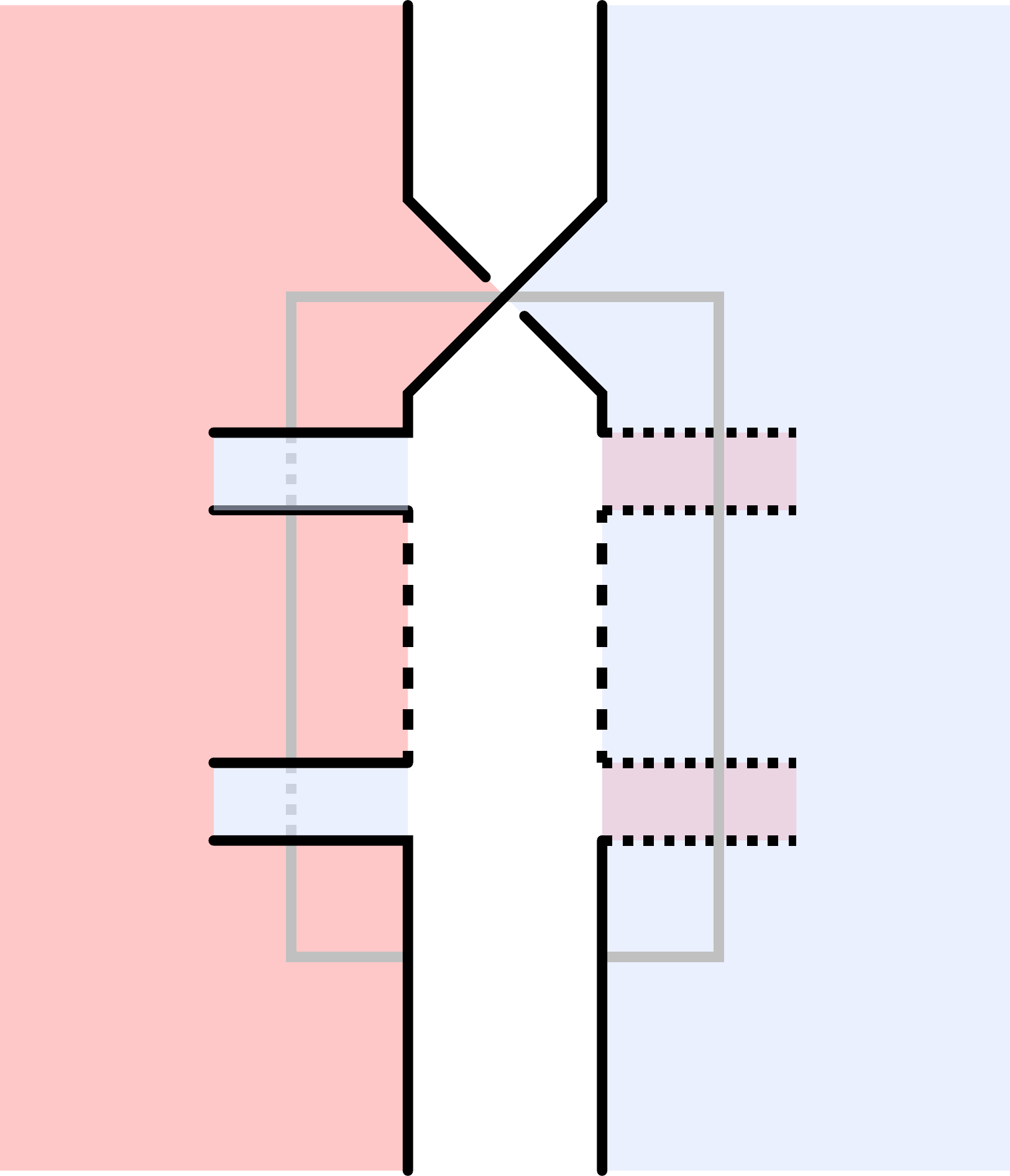}
\end{center}
\caption{
Inserting a positive crossing next to another one by positive Hopf plumbing (generalizing \cref{fig:plumb}).
\newline
\textbf{(a)} Local picture of $\Sigma'(D)$ containing the two ribbons corresponding to two crossings $c_1$ and $c_2$ that are next to each other. Note that the central white region could contain infinity.
\newline
\textbf{(b)} The result of plumbing a positive Hopf band in (c).
\newline
\textbf{(c)} A closed interval (gray) in $\Sigma'(D')$ along which a positive Hopf band gets plumbed to the blue side.
}
\label{fig:nexttoplumbing}
\end{figure}
This is due to the $z$--coordinate convention of nested disks assuring that all the ribbons corresponding to crossings on $k$ and $k'$ to the other side than the $c_i$ are to the positive (red) side of the disks corresponding to $k$ and $k'$.
Then, the surface can be isotoped to be the result (see \cref{fig:nexttoplumbing}(b)) of plumbing a positive Hopf band to the negative (blue) side of $\Sigma'(D')$ (see \cref{fig:nexttoplumbing}(c)), where $D'$ is the diagram obtained by deleting the crossing $c_i$ in $D$ that is depicted at the bottom of \cref{fig:nexttoplumbing}(a).

If instead $k$ and $k'$ are $s_1$ and $s_2$, then the crossings $c_1$ and $c_2$ lie in $U_1$ or $U_2$. If they lie in $U_1$, the situation is again exactly as depicted in \cref{fig:nexttoplumbing}(a).
If instead the crossings $c_1$ and $c_2$ lie in $U_2$, then $\Sigma'(D)$ can be isotoped to look like \cref{fig:nexttoplumbing}(a) by locally pulling the disks corresponding to $s_1$ and $s_2$ apart, so the first is no longer above the second.
Note that, different from the previous cases, the positive side of $\Sigma'(D)$ is depicted as blue and the plumbing of the positive Hopf band happens to the positive (blue) side of $\Sigma'(D')$. Again, here $D'$ is the appropriate diagram obtained from $D$ by deleting either $c_1$ or~$c_2$.

\textbf{Case when $k$ and $k'$ are nested.}

First remark that $k$ or $k'$ is a Seifert circle since $s_1$ and $s_2$ are not nested.
Second, note that either the two $c_i$ and the two $s_i$ all lie to the same side of $k$, or they all lie to the same side of $k'$.
All in all, we suppose w.l.o.g. that $k$ is a Seifert circle and
the $s_i$ and the $c_i$ all lie to the same side of $k$.

We now split $\Sigma'(D)$ as a Murasugi sum along $k$.
Let $D_{-}$ and $D_+$ be the link diagrams so that $D=D_{-} \cup D_+$ and $D_{-} \cap D_+ = k$ (as in \cref{lem:intext}), where we let $D_{-}$ be the link diagram that contains the $s_i$ and the $c_i$.
The simple case that $D_+$ consists only of $k$ and $D_-=D$ is possible.
The Seifert surface $\Sigma'(D)$ is a Murasugi sum of $\Sigma(D_+)$
and $\Sigma'(D_-)$. Since $D_+$ has no negative crossings, $\Sigma(D_+)$ is a quasipositive Seifert surface and, thus, $\Sigma'(D)$ is a quasipositive Seifert surface if and only if $\Sigma'(D_-)$ is by \cref{eq:murasugi}.

We now argue that $\Sigma'(D_-)$ arises by positive Hopf plumbing on $\Sigma'(D_-')$, where $D_-'$ is a diagram obtained from deleting one of the $c_i$ in $D_-$. For this we note that the Seifert surface $\Sigma'(D_-)$ can be isotoped (by folding the disk corresponding to either $k$ or $k'$, which ever contains the other, along the part of its boundary connecting the two ribbons corresponding to $c_1$ and $c_2$) to look like \cref{fig:nexttoplumbing}(a).

In fact, the situation is necessarily simpler as depicted in \cref{fig:nexttoplumbing}(a): on the disk corresponding to $k$ there will be no ribbons leaving between $c_1$ and $c_2$ on either side (red or blue). In other words, the situation is as depicted in \cref{fig:nexttoplumbing}(a) to one side and as depicted in \cref{fig:plumb}(c) on the other side.

So then, as before, $\Sigma'(D_-)$ is the result (see \cref{fig:nexttoplumbing}(b)) of plumbing a positive Hopf band to $\Sigma'(D_-')$ (see \cref{fig:nexttoplumbing}(c)), where $D_-'$ is the appropriate diagram obtained from $D_-$ by deleting one of the $c_i$. We note that both plumbing to the positive side and plumbing to the negative side can occur.

Finally, we set $D'$ to be the union of $D_-'$ and $D_+$. The Seifert surface $\Sigma'(D')$ is a Murasugi sum of $\Sigma(D_+)$ and $\Sigma'(D_-')$ and, thus, $\Sigma'(D')$ is a quasipositive Seifert surface if and only if $\Sigma'(D_-')$ is (by \cref{eq:murasugi}). Therefore, we conclude that $\Sigma'(D)$ is a quasipositive Seifert surface if and only if $\Sigma'(D')$ is, as desired.
\end{proof}

\section{Canonical quasipositive surfaces}
\label{sec:canonical}
Let us start with some graph theoretic concepts.

\begin{definition}
	\label{def:graph_concepts}
	$\,$

\begin{itemize}
\item A \emph{path} $P$ is a sequence $e_1, \ldots, e_n$ of distinct edges in which
$e_i$ has vertices $v_i^1$ and $v_i^2$ such that $v_i^2 = v_{i+1}^1$ and such that
every vertex appears at most twice as endpoint of an edge of $P$.
\item The \emph{length} of such a path $P$ is denoted by $\ell(P) = n$.
\item A \emph{cycle} $C$ is a path as above with $v_n^2 = v_1^1$.
\item A \emph{region} of a plane graph $G$ is a connected component of $\mathbb{R}^2 \setminus G$.
\item A graph $G$ is \emph{2--connected} if it has at least three vertices,
is connected, and the result of removing any vertex is again connected.
\item A \emph{weighted graph} is a graph in which each edge carries either the weight $+1$ or the weight $-1$.  For a collection $E$ of edges of a weighted graph, we denote by $w(E) \in \Z$ the \emph{total weight} of $E$, i.e.\ the  sum of the weights of the edges in $E$.
\end{itemize}
\end{definition}
Our main theorem of this section is the following.

\thmB*

\begin{proof}
A cycle $C$ of $\Gamma(D)$ lifts to a non-null-homologous unknot in $\Sigma(D)$ with framing~$w(C)$.
A tubular neighborhood of that unknot in $\Sigma(D)$ is an annulus with $w(C)$ full twists, and an incompressible subsurface of $\Sigma(D)$. So if $\Sigma(D)$ is quasipositive, then $w(C) > 0$ follows from \cref{eq:subsurface}.
This establishes the necessity of the cycle condition for quasipositivity.

To see that that the cycle condition for quasipositivity is sufficient,
suppose some diagram $D$ has $\Gamma(D)$ satisfying the hypothesis of \cref{thm:2}.

If $D$ is a split diagram, then $\Gamma(D)$ is not connected and the cycle condition can be checked on each connected component individually. Therefore we may and do assume that $D$ is non-split.

If there is a Seifert circle $k$ in $D$ that has non-empty interior and exterior, then $\Sigma(D)$ may be expressed as the Murasugi sum of some $\Sigma(D_i)$ and $\Sigma(D_e)$ (see \cref{lem:intext}). Since $\Gamma(D_i)$ and $\Gamma(D_e)$ are both subgraphs of $\Gamma(D)$, we conclude that \cref{thm:2} follows from considering diagrams where each Seifert circle either has empty interior or empty exterior.

By moving infinity we move to a different diagram but with an isotopic canonical surface.  So, by possibly moving infinity, we may and do assume that every Seifert circle of $D$ has empty interior.

Hence we have reduced the proof to \cref{prop:main_theorem}.
\end{proof}

\begin{proposition}
	\label{prop:main_theorem}
	If $D$ is a non-split link diagram such that every Seifert circle of $D$ has empty interior, and such that all cycles of $\Gamma(D)$ have positive total weight, then $\Sigma(D)$ is quasipositive.
\end{proposition}
Recall from the paragraph after \cref{thm:2} in \cref{sec:almost1} that if $D$ is a link diagram with all Seifert circles having empty interior, then its Seifert graph $\Gamma(D)$ is naturally a plane graph, while also being bipartite and weighted.  All graphs considered in this section will be bipartite weighted plane graphs.

\begin{proof}[Proof of \cref{prop:main_theorem}]
Let $D$ be a link diagram satisfying the hypothesis of the proposition. Since $D$ is non-split, $\Gamma(D)$ is connected.

Suppose that $D$ is a connected sum of diagrams $D_1$ and $D_2$, each with at least one crossing.  Then $\Sigma(D)$ is a connected sum (a special case of a Murasugi sum) of $\Sigma(D_1)$ and $\Sigma(D_2)$ and it suffices to consider the summands by \cref{eq:murasugi}.  Therefore we only consider diagrams that are not such connected sums.

If $D$ has only one Seifert circle, $\Sigma(D)$ is a disk, which is quasipositive.  Similarly if $D$ has only two Seifert circles and only one crossing.

If $D$ has two Seifert circles and $n \geq 2$ crossings, then each crossing is positive since otherwise the cycle positivity condition of $\Gamma(D)$ would be violated.  Then we see that
$\Sigma(D)$ is the fiber surface of the positive $(2,n)$--torus link, which is quasipositive.
This case will be the root case of a proof by induction.

We assume now that $D$ has at least three Seifert circles. Since $D$ is not a non-trivial connected sum, $\Gamma(D)$ is 2--connected (see \cref{def:graph_concepts}).
Then it is a straightforward graph theoretic result about 2--connected plane graphs that the boundary of each region of $\R^2 \setminus \Gamma(D)$ is a cycle. We will call such a cycle \emph{boundary cycle}. 

Our strategy is a proof by induction over the following measure of complexity of~$\Gamma(D)$.
Suppose $x = (x_1, x_2, x_3,\ldots )$ and $y = (y_1, y_2, y_3, \ldots)$ are two infinite sequences of integers with only finitely many non-zero integers.  Define $x > y$  iff the rightmost non-zero entry of $x-y$ is positive.
Let $f_i$ be the number of boundary cycles of length $2i$. We define the infinite sequence
\[
f(\Gamma(D)) = (f_1, f_2, f_3, \ldots).
\]
Given a link diagram $D$ satisfying the hypothesis of \cref{prop:main_theorem}, our idea is to produce a new link diagram $D'$ also satisfying the hypothesis.  Furthermore we aim to do this so that $f(\Gamma(D')) < f(\Gamma(D))$ and so that $\Sigma(D)$ is a quasipositive surface if $\Sigma(D')$ is a quasipositive surface.
Having already verified the root case that all boundary cycles have length 2 (which implies having two Seifert circles), the induction will give us the result.

If $\Gamma(D)$ contains a vertex of degree $2$ adjacent to a positive and a negative edge, then we have $\Sigma(D) = \Sigma(D')$ where $D'$ is obtained from $D$ by removing two crossings via a Reidemeister II move.  Furthermore $D'$ satisfies the hypothesis of \cref{prop:main_theorem} and has $f(\Gamma(D')) < f(\Gamma(D))$.
So we may and do assume that $\Gamma(D)$ contains %
no degree $2$ vertices adjacent to both a positive and negative edge.

Let us now introduce a new move, which generalizes (\cref{c7}) from the proof of \cref{thm:1}.
Suppose $v, w$ are vertices of $\Gamma(D)$ on the boundary $C$ of a region of $\R^2 \setminus \Gamma(D)$. Let $d$ be the distance (lengthwise, not weighted)
between $v$ and $w$ along $C$ and suppose $d \geq 2$.
We now describe a diagram $D'$ obtained from $D$.  It is enough to describe $\Gamma(D')$, which is obtained from $\Gamma(D)$ by adding
a chord consisting of a path of $(d - 2)$ positive edges between $v$ and $w$ inside of the region. In the special case
of $d = 2$, adding a chord of length $0$ is understood as merging $v$ and~$w$.
The two crucial observations are:
\begin{itemize}
\item We have that $f(\Gamma(D')) < f(\Gamma(D))$ because the regions of $\R^2 \setminus \Gamma(D')$ correspond to those of $\R^2 \setminus \Gamma(D)$, except for the
region with boundary $C$, which is split into two regions, each of them with strictly fewer edges than $C$.
\item The surface $\Sigma(D)$ is an incompressible subsurface of $\Sigma(D')$.
\end{itemize}
So to conclude the proof of the proposition, it suffices to show that this move is always possible in such a
way that $D'$ still satisfies the hypothesis of the proposition.
Note that the only cycles of $\Gamma(D')$ not occurring as cycles of $\Gamma(D)$ are those that pass through the new chord.
So it suffices to pick $v, w$ and $C$ such that
any path $Q$ in $\Gamma(D)$ between $v$ and $w$ satisfies $w(Q) + d - 2 \geq 2$.
That this is always possible is the contents of \cref{prop:weight_4_region_exists_really,prop:weight_4_are_splittable}.%
\end{proof}

\begin{definition}
	\label{def:property_starfish}
	We say that a link diagram $D$ has property (*) if it satisfies the following.\\
	\begin{itemize}
		\item All Seifert circles of $D$ have empty interior.
		\item All cycles of $\Gamma(D)$ have positive total weight.
		\item $\Gamma(D)$ is 2--connected (see \cref{def:graph_concepts}).
		\item $\Gamma(D)$ contains no degree $2$ vertices adjacent to both a positive and a negative edge.
	\end{itemize}
\end{definition}

\begin{definition}
	\label{def:splittable_boundary}
	Suppose that $D$ has property (*) and $C$ is the boundary cycle of a region of $\R^2 \setminus \Gamma(D)$.  We say that $C$ is \emph{splittable} if there exist vertices $v$ and $w$ of $C$, distance $d \geq 2$ apart on $C$, such that every path $Q$ in $\Gamma(D)$ connecting $v$ to $w$ satisfies $w(Q) \geq 4-d$.
\end{definition}

\begin{proposition}\label{prop:weight_4_region_exists_really}
	Suppose that $D$ has property~(*).  Then there is a region of $\R^2 \setminus \Gamma(D)$ whose boundary cycle $C$ has $w(C) \geq 4$.
\end{proposition}

We postpone the proof of this proposition to the end of the section.

\begin{proposition}
	\label{prop:weight_4_are_splittable}
	Suppose that $D$ has property (*) and $C$ is the boundary cycle of a region of $\R^2 \setminus \Gamma(D)$ with $w(C) \geq 4$.  Then $C$ is splittable.
\end{proposition}

\begin{proof}
	The proof is divided into two cases given as \cref{prop:not_positive_region,prop:positive_region}.
\end{proof}

\begin{lemma}
	\label{prop:not_positive_region}
Suppose that $D$ has property (*), and that $C$ is the boundary cycle of a region of $\R^2 \setminus \Gamma(D)$ with $w(C) \geq 4$ and at least one edge of $C$ negative.  Then $C$ is splittable.

\end{lemma}

\begin{lemma}
	\label{prop:positive_region}
	Suppose that $D$ has property (*), and that $C$ is the boundary cycle of a region of $\R^2 \setminus \Gamma(D)$ with $w(C) \geq 4$ and every edge of $C$ positive.  Then $C$ is splittable.

	\end{lemma}

For the proof of \cref{prop:not_positive_region,prop:positive_region} we first collect some straightforward facts, without proof, in the following lemma.

\begin{lemma}
	\label{lem:some_useful_facts}
	We have the following.
	\begin{enumerate}
		\item Suppose that $H$ is a graph and $v$ and $w$ are distinct vertices of $H$.  Then paths from $v$ to $w$ are exactly the minimal subgraphs of $H$ in which $v$ and $w$ are the only vertices with odd degree.
		\item If $H$ is a graph with vertices of only even degree then its set of edges can be written as a disjoint union of cycles.
		\item Suppose that $H$ is a graph with exactly two vertices $v$ and $w$ of odd degree and that $P$ is any path from $v$ to $w$ (such a $P$ exists by the first part of this lemma).  Then removing the set of edges of $P$ from $H$ leaves a graph whose set of edges is a disjoint union of cycles.\qed
	\end{enumerate}
\end{lemma}

\begin{proof}[Proof of \cref{prop:not_positive_region}]
Let us pick $v$ and $w$ as follows. Walking around $C$,
pick $v$ such that the next edge is positive and the one immediately after is negative,
and continue walking until the next positive edge, and call its farther vertex $w$.
Then one has walked along a path $P$ with
$w(P) = 4 - \ell(P)$. The other path $P'$ in $C$ between $v$ and $w$ satisfies $w(P') + w(P) \geq 4$,
and so $\ell(P') \geq w(P') \geq 4 - w(P) = \ell(P)$.  Hence $P$ is not the longer path between $v$ and $w$ around $C$, and
$d = \ell(P)$.

For every path $Q$ in $\Gamma(D)$ between $v$ and $w$, we must show that
$w(Q) \geq 4-d$, which is equivalent 
to showing $w(Q) \geq w(P)$ since $4-d=4-\ell(P)=w(P)$.

So let $Q$ be such a path.  Note that $(P \cup Q) \setminus (P \cap Q)$ is a union of cycles (see \cref{lem:some_useful_facts}(2)), say $Z_1, \ldots, Z_n$.  Let us write $Z_i = P_i \sqcup Q_i$ for a decomposition of each $Z_i$ into sets of edges $P_i \subset P$ and $Q_i \subset Q$.
Then we have that
\begin{align*}
w(P) &= w(P_1) + \ldots + w(P_n) + w(P \cap Q),\\
w(Q) &= w(Q_1) + \ldots + w(Q_n) + w(P \cap Q).
\end{align*}
So we shall be done if we can show that $w(Q_i) \geq w(P_i)$ for all $i$.

Since $P$ contains only two positive edges, we must have $w(P_i) \leq 2$, with equality only when $P_i$ consists of exactly the two positive edges of $P$ (in other words the first and last edge of $P$).  Since the path $Q_i$ of course cannot enter the region bounded by $C$ the only way that $w(P_i) = 2$ can happen is if $Q_i$ consists of a path connecting $v$ to $w$ and another path connecting the other two endpoints of the first and last edge of~$P$.  But $Q$ is a path connecting $v$ to $w$ and $Q_i \subset Q$, so we have a contradiction
by \cref{lem:some_useful_facts}(1).
Therefore we must have $w(P_i) \leq 1$ for all $i$.  Further note that $w(Q_i) + w(P_i) = w(Z_i) \geq 2$ since $Z_i$ is a cycle.  Hence we can conclude that $w(Q_i) \geq w(P_i)$ for all $i$.
\end{proof}

A heuristic important for our proof of \cref{prop:positive_region} is that if two paths of a weighted graph intersect, one can resolve them and obtain two new paths which have the same total weight.  This heuristic turns out, when formalized in the following lemma, only to yield an inequality rather than an equality.

\begin{lemma}
	\label{lem:intersecting_lemma}
	Let $D$ be a diagram which has property (*) and consider the boundary cycle $C$ of a region of $\R^2 \setminus \Gamma(D)$.  Suppose that $\ell(C) \geq 4$, and let $v_i$ be distinct vertices of $C$ for $i \in \Z/4$, occurring in the cyclic ordering around $C$.  Suppose that $P_{02}$ and $P_{13}$ are paths from $v_0$ to $v_2$ and from $v_1$ to $v_3$ respectively.
	
	Then for some $i \in \{0,1\}$ there exists a path $P_{i,i+1}$ from $v_i$ to $v_{i+1}$ and a path $P_{i+2, i+3}$ from $v_{i+2}$ to $v_{i+3}$ such that
\[ w(P_{i,i+1}) + w(P_{i+2, i+3}) \leq w(P_{02}) + w(P_{13}) {\rm .} \]
\end{lemma}

\begin{proof}
	Consider the subgraph $H$ of $\Gamma(D)$
	\[ H = (P_{02} \cup P_{13}) \setminus (P_{02} \cap P_{13}) {\rm .} \]
	The vertices in $H$ of odd degree are exactly the vertices $v_i$ for $i \in \Z / 4$.  Let $\widetilde{H}$ be the subgraph of $H$ obtained by removing those connected components of $H$ not containing any of the vertices $v_i$. %
	Again the vertices of $\widetilde{H}$ of odd degree are exactly the vertices $v_i$.  Therefore $\widetilde{H}$ is either connected or has exactly two components each containing two of the vertices $P_i$.
Any path connecting $v_0$ and $v_2$ must intersect any path connecting $v_1$ and $v_3$.
Hence, it cannot be the case that there are two components of which one contains $v_0$ and $v_2$ and while the other contains $v_1$ and $v_3$.  Therefore $v_0$ is in the same component as $v_1$ or as $v_3$.  Let us assume the former for now.
	
	Let $Q$ be a path in $\widetilde{H}$ connecting $v_0$ to $v_1$.  Note that $Q$ is a subgraph of $H$ and so, by construction, each edge of $Q$ occurs either in $P_{02}$ or in $P_{13}$ but not in both.
	Therefore we have
	\begin{align*}
		w(P_{02}) + w(P_{13}) =
		w((P_{02} \cup Q) \setminus (P_{02} \cap Q)) + w((P_{13} \cup Q) \setminus (P_{13} \cap Q)) {\rm .}
	\end{align*}
	
	Now $v_1$ and $v_2$ are exactly the vertices of $(P_{02} \cup Q) \setminus (P_{02} \cap Q)$ of odd degree.  Therefore $(P_{02} \cup Q) \setminus (P_{02} \cap Q)$ can be written as the disjoint union of a path $P_{12}$ from $v_1$ to $v_2$ and some cycles (by \cref{lem:some_useful_facts}).  Since by assumption each cycle has weight $\geq 2$, we must have that
	\[ w(P_{12}) \leq w((P_{02} \cup Q) \setminus (P_{02} \cap Q)) {\rm .} \]
	Similarly there is a path $P_{30}$ from $v_3$ to $v_0$ satisfying
	\[ w(P_{30}) \leq w((P_{13} \cup Q) \setminus (P_{13} \cap Q)) {\rm .} \]
	
	Hence we have
	\begin{align*}
		w(P_{12}) + w(P_{30}) &\leq w((P_{02} \cup Q) \setminus (P_{02} \cap Q)) + w((P_{13} \cup Q) \setminus (P_{13} \cap Q)) \\
		&= 	w(P_{02}) + w(P_{13}) {\rm .}
	\end{align*}
	The other case follows from making the assumption that $v_0$ is in the same component of $\widetilde{H}$ as $v_3$.
\end{proof}

\begin{lemma}
	\label{lem:paths_of_low_weight_are_verboten}
	Suppose that $D$ has property (*) and that $C$ is the boundary of a region of $\R^2 \setminus \Gamma(D)$ such that all edges of $C$ have weight $+1$.  Let $v$ and $w$ be vertices of $C$ such that the shortest path in $C$ between $v$ and $w$ has length $d \geq 1$.  Then there is no path between $v$ and $w$ of weight less than $2-d$.

\end{lemma}

\begin{proof}
	Let $P$ be a path between $v$ and $w$.  Let us write $Q$ for a path of length $d$ contained in $C$ connecting $v$ to $w$.
	Consider $P \cap Q$.  Each connected component of $P \cap Q$, since it is a subgraph of $Q$, is a path in $C$.  Furthermore %
	the set of edges of $(P \cup Q) \setminus (P \cap Q)$ can be written as a disjoint union of cycles.  Let us write $c$ for the number of cycles in such a decomposition of $(P \cup Q) \setminus (P \cap Q)$.
	
	In the case that $c=0$ then $P=Q$ and $w(Q) = w(P) = d \geq 2-d$.
	
	In the case that $c \geq 1$ then
	\begin{align*}
\pushQED{\qed}
	w(P) &= w((P \cup Q) \setminus (P \cap Q)) + 2w(P \cap Q) - w(Q) \\
	& \geq 2c + 2w(P \cap Q) - d \geq 2c - d \geq 2 - d {\rm .}
\qedhere\popQED
	\end{align*}
\renewcommand{\qedsymbol}{}
\end{proof}

Now we are in a position to give the proof of \cref{prop:positive_region}, thus establishing \cref{prop:weight_4_are_splittable}.

\begin{proof}[Proof of \cref{prop:positive_region}]
	Let us proceed by contradiction.  So assume that $C$ is a boundary cycle with no negative edges, of total weight $w(C)=\ell(C) =2n \geq 4$,
and assume $C$ is not splittable. That is to say that for any pair of vertices $v, w$ on $C$ of distance $d$ along $C$,
there is a path $Q$ in $\Gamma(D)$ with $w(Q) < 4 - d$, thus $2n \leq 2 - d$.

Pick $v_i$ for $i\in \Z/4$ on $C$ in the cyclic ordering, such that the distances between $v_i$ and $v_{i+1}$
are $1, 1, n-1$ and $n-1$ for $i = 0, 1, 2, 3$. Thus there are paths $P_{02}$ and $P_{13}$ from $v_0$ to $v_2$ and from $v_1$ to $v_3$,
respectively, such that $w(P_{02}) \leq 0$ and $w(P_{13}) \leq 2 - n$.
By \cref{lem:paths_of_low_weight_are_verboten} it follows that in fact $w(P_{02}) = 0$ and $w(P_{13}) = 2 - n$.
By \cref{lem:intersecting_lemma}, there are paths $P_{i,i+1}$ and $P_{i+2,i+3}$ for some $i\in\{0,1\}$ such that
$w(P_{i,i+1}) + w(P_{i+2,i+3}) \leq 2 - n$.
This contradicts the fact that by positivity of cycle weights, $w(P_{i,i+1}) > -1$ and $w(P_{i+2,i+3}) > 1 - d$
and thus $w(P_{i,i+1}) \geq 1$ and $w(P_{i+2,i+3}) \geq 3 - d$.
\end{proof}

Finally, we turn to \cref{prop:weight_4_region_exists_really}, to whose proof we devote the remainder of this section.

\begin{proof}[Proof of \cref{prop:weight_4_region_exists_really}]
We divide the proof of this proposition into \cref{lem:weight_4_region_exists,lem:wicked_vertices_exist}.
\end{proof}

Let us first prove the following.
\begin{lemma}
	\label{lem:posvert}
Suppose that $D$ is a diagram which has property~(*).  Suppose further that the boundaries of all regions of $\R^2 \setminus \Gamma(D)$ have total weight $2$.  Then there is a \emph{positive} vertex of $\Gamma(D)$, in other words a vertex not adjacent to a negative edge.
\end{lemma}
\begin{proof}
As before, let $f_i$ be the number of boundary cycles of length $2i$.
Let $f$ be the total number of regions, $e$ the number of edges, $e_-$ the number of negative edges and $v$ the number of vertices. Then we have
\[
f  = \sum_{i=2}^{\infty} f_{2i} \quad {\rm and} \quad e  = \sum_{i=2}^{\infty} if_{2i} \quad \text{so that}\quad
v = 2 + \sum_{i=2}^{\infty} (i-1)f_{2i} {\rm .}
\]
Then, since every region of $\R^2 \setminus \Gamma(D)$ has two more positive edges in its boundary than negative edges, we have
\[
\pushQED{\qed}
e_- = \sum_{i=2}^{\infty} (i - 1) f_{2i} / 2 \quad \text{and so}\quad v > 2e_- {\rm .}
\qedhere\popQED
\]
\renewcommand{\qedsymbol}{}
\end{proof}

\begin{definition}
	\label{def:wicked}
	Let $D$ be a diagram which has property (*) and let $v$ be a positive vertex of $\Gamma(D)$.
	Let $C_1, \ldots, C_{n}$ be the boundaries of those regions of $\R^2 \setminus \Gamma(D)$ adjacent to $v$.  %
	We say that $v$ is a \emph{wicked} positive vertex if $C_i \cap C_j$ is connected for all $i,j$.
\end{definition}

\begin{lemma}
	\label{lem:weight_4_region_exists}
	Suppose that $D$ is a diagram which has property (*) and which contains a wicked positive vertex.  Then $\R^2 \setminus \Gamma(D)$ contains a region whose boundary cycle has weight at least $4$.
\end{lemma}

\begin{proof}
	For a contradiction, suppose that $D$ is a diagram which has property (*) and in which there is no region of $\R^2 \setminus \Gamma(D)$ whose boundary has weight $4$ or greater, and suppose that $v$ is a wicked positive vertex of $\Gamma(D)$.
	
	Let $C_1, \ldots, C_{n}$ be the boundaries of those regions of $\R^2 \setminus \Gamma(D)$ adjacent to $v$, in counterclockwise order around $v$, where the subscripts are considered modulo~$n$.
If all of the $C_i$ had length 2, then $v$ would have $n$ edges adjacent to a vertex $w$, and no further edges, which would contradict property~(*).
	Since ${C_i} \cap {C_{i+1}}$ is connected, it is a path starting at $v$ (or containing $v$ in case $n=2$).  Since $D$ has property~(*), all these paths contain
	only positive edges, and thus have positive total weight.
	Now, using that not all $C_i$ have length 2, the edges of
	\[ (C_1 \cup \cdots \cup C_n) \setminus \bigcup_i (C_i \cap C_{i+1}) \]
	form a cycle $C$ with
	\begin{align*}
	w(C) & = \sum_i w(C_i) - 2\sum_i w(C_i \cap C_{i+1}) \\
	& = 2n - 2\sum_i w(C_i \cap C_{i+1}) \\
	& \leq 2n - 2n = 0 {\rm .}
	\end{align*}
	But this contradicts the property~(*).
\end{proof}

\begin{lemma}
	\label{lem:wicked_vertices_exist}
	If $D$ is diagram which has property (*) then either $\Gamma(D)$ has a wicked positive vertex, or $\R^2 \setminus \Gamma(D)$ contains a region whose boundary cycle has weight at least $4$.
\end{lemma}

\begin{proof}[Proof of \cref{lem:wicked_vertices_exist}]
	Suppose for a contradiction that there exists at least one diagram with property (*) whose Seifert graph contains no positive wicked vertices and no boundaries of weight at least $4$.  Consider such diagrams with the minimal number of positive vertices, and let $D$ be one of these with the minimal number of crossings.
	
	We know by \cref{lem:posvert} that $D$ has a positive vertex; let us call it $v$.
	Let $C_1, \ldots, C_{n}$ be the boundaries of those regions of $\R^2 \setminus \Gamma(D)$ adjacent to $v$, in counterclockwise order around $v$, where the subscripts are considered modulo $n$.
	Then, since $v$ is not wicked by assumption, for some $i \not= j$ we have that $C_i \cap C_j$ has $m \geq 2$ components (note that one of these components contains the positive vertex~$v$).  We write $R_i$ and $R_j$ for the closed bounded regions of $\R^2$ whose boundaries are $C_i$ and $C_j$ respectively.  Then $\R^2 \setminus (R_i \cup R_j)$ has $m$ components $B_1, \ldots, B_n$ and the boundary of each $\overline{B_k}$ is a cycle in $\Gamma(D)$.  Let us write $Z_k$ for the boundary of $\overline{B_k}$, and $P_1, P_2, \ldots, P_n$ for the paths which are the components of $R_i \cap R_j$.

	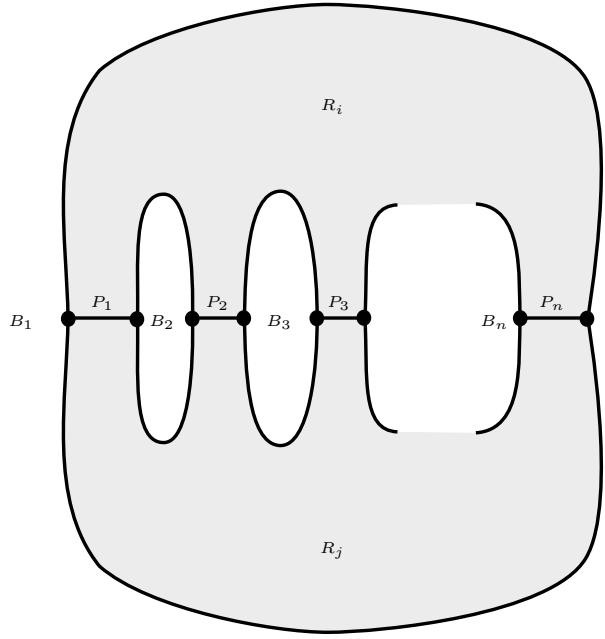
\begin{figure}[hbt]
		\labellist
		\tiny
		\pinlabel {$R_i$} at 640 1250
		\pinlabel {$R_j$} at 640 200
		\pinlabel {$B_1$} at -90 740
		\pinlabel {$B_2$} at 240 740
		\pinlabel {$B_3$} at 515 740
		\pinlabel {$B_n$} at 1020 740
		\pinlabel {$P_1$} at 100 785
		\pinlabel {$P_2$} at 370 785
		\pinlabel {$P_3$} at 655 785
		\pinlabel {$P_n$} at 1155 785
		\endlabellist
		\includegraphics[scale=0.16]{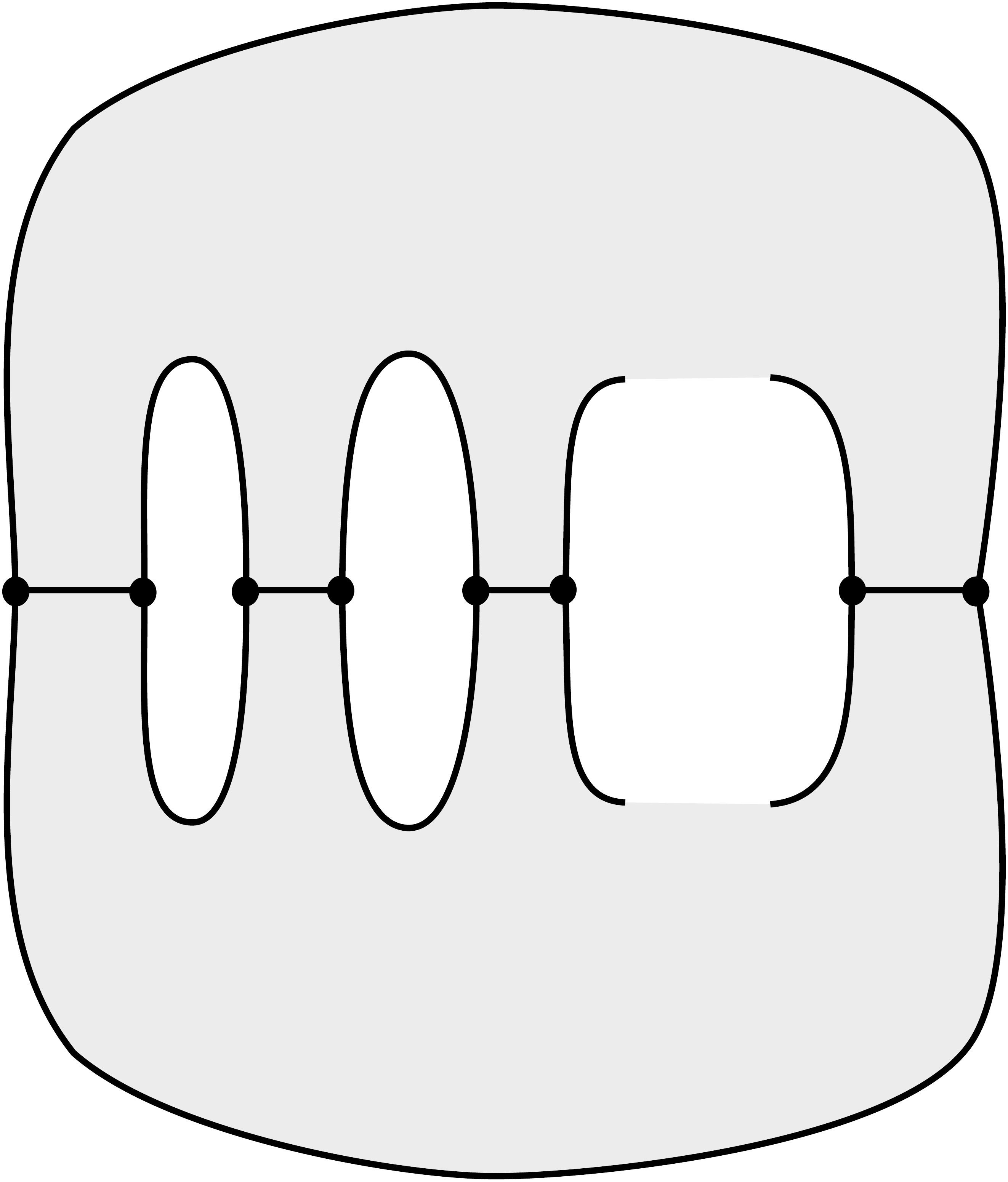}
		\caption{A diagram of the situation of \cref{lem:wicked_vertices_exist}.}
		\label{fig:CiCj}
	\end{figure}
	
	The situation is illustrated in \cref{fig:CiCj}.  Note that the interiors of the regions $R_i$ and $R_j$ do not contain any vertices or edges of $\Gamma(D)$.  Note also that some paths $P_k$ could consist of single vertices.  Note further that the vertices of each $P_k$ which are not endpoints of $P_k$ have degree $2$ in $\Gamma(D)$.  Since $\Gamma(D)$ has property (*) it follows that no $P_k$ has both positive and negative edges.
	
	Now for any $k$ consider the subgraph of $\Gamma(D)$ that lies within $\overline{B_k}$.  This is the Seifert graph of a diagram $D'$ that also satisfies property~(*), apart from possibly containing a degree $2$ vertex adjacent to both a positive and a negative edge.  By Reidemeister II moves $D'$ may be converted to a diagram $D''$ with property (*) such that $D''$ has the same number of Seifert circles as $D'$ and possibly fewer crossings.  Hence by induction $\R^2 \setminus \Gamma(D'')$ contains a region whose boundary has weight $4$ or more.  However the regions of $\R^2 \setminus \Gamma(D'')$ are in an obvious correspondence with those of $\R^2 \setminus \Gamma(D')$ under which the weights the boundary cycles are invariant.  Hence $\R^2 \setminus \Gamma(D')$ contains a region whose boundary has weight $4$ or more.  We conclude that either $\R^2 \setminus \Gamma(D)$ does as well (and we are done) or that $w(Z_k) \geq 4$.
	
	We assume for a contradiction that we have $w(C_k) = 2$ for all $C_k$.  Then we have
	\begin{align*}
	2(w(P_1) + \cdots + w(P_n)) &= 2w(C_i \cap C_j) \\
	 &= w(C_i) + w(C_j) - (w(Z_1) + \cdots + w(Z_n)) \\
	&= 4 - (w(Z_1) + \cdots + w(Z_n)) \\
	&\leq 4 - 4n \leq -2n {\rm ,}
	\end{align*}
	where the last inequality is because $n \geq 2$.
	
	Since no $P_k$ has both positive and negative edges, and at least one of the $P_k$, say $P_\alpha$, (that containing the vertex $v$) contains no positive edge, it follows that at least one of the $P_i$ contains two consecutive negative edges.  Let $w$ be the midpoint of these two consecutive edges.  Note that there is an embedded circle in $\R^2$ which meets $\Gamma(D)$ only at $v$ and $w$.  Flyping along this circle to move one of the negative edges to lie adjacent to $v$ creates a new graph which is the Seifert graph $\Gamma(D')$ of a diagram $D'$.
	
	Now, in the case that the path $P_\alpha$ contains at least one edge, $D'$ admits simplification via a Reidemeister II move to a new diagram $D''$ such that $D''$ has property~(*). %
	Furthermore $\Gamma(D'')$ has either the same number or one fewer positive vertices than~$\Gamma(D)$.  Hence by assumption, $\Gamma(D'')$ must contain a wicked positive vertex and so by \cref{lem:weight_4_region_exists} contains a region whose boundary cycle has weight $4$.
	
	But there is an obvious correspondence between the regions of $\R^2 \setminus \Gamma(D'')$ and those of $\R^2 \setminus \Gamma(D)$, and the weight of the boundary cycle of each region is preserved under this correspondence.  Hence we get a contradiction in this case.
	
	Finally consider the remaining case that the path $P_\alpha$ contains no edges.  In this case we have that $\Gamma(D')$ has property (*), has no wicked positive vertices, and has one fewer positive vertices than does $\Gamma(D)$.  Hence by assumption $\R^2 \setminus \Gamma(D')$ must contain a region whose boundary cycle has weight at least $4$.  But again, there is a weight-preserving correspondence between the regions of $\R^2 \setminus \Gamma(D')$ and those of $\R^2 \setminus \Gamma(D)$.  Hence we have a contradiction.
\end{proof}

\section{Applications}
\label{sec:corollaries}
This section contains the proofs of the applications of \cref{thm:2} mentioned in the introduction.
\corC*
\begin{proof}
For the `only if' direction (for which the assumption that $\Sigma$ is isotopic to a canonical surface is not necessary),
assume $\Sigma$ is quasipositive
and let $\gamma\subset\Sigma$ be an unknot that does not bound a disk in $\Sigma$,
with framing $k$ induced by~$\Sigma$.
Let $B$ be an annular neighborhood of $\gamma$ in~$\Sigma$. As an incompressible subsurface of $\Sigma$,
$B$ is itself quasipositive. But as an unknotted band, $B$ is quasipositive only if its core
curve $\gamma$ has negative induced framing by $B$ (which equals the framing induced by $\Sigma$).
It follows that $k < 0$, 
concluding the proof of the `only if' direction.

For the `if' direction, let $\Sigma$ be a canonical Seifert surface. By \Cref{thm:2}, it suffices to check that all cycles in the Seifert graph of $\Sigma$ have strictly positive total weight. For a cycle $c$ in the Seifert graph, we let $A_c$ be the embedded annulus in $\Sigma$ given by the union of Seifert circles and half-bands that correspond to the vertices and edges that make up $c$.
Note that $A_c$ is an unknotted annulus and incompressible in $\Sigma$.
So, by assumption, $A_c$ has strictly negative framing. Twice the framing of $A_c$ equals minus the total weight of $c$. To see this, note that the framing is calculated as the linking of the two boundary components of $A_c$ where the orientation is reversed on one component, hence every positive crossing traversed by $A_c$ contributes $-\frac{1}{2}$ to the framing, while it contributes $+1$ to the total weight of $c$ (and analogously for negative crossings).
With this we established that all cycles have strictly positive total weight, as required.
\end{proof}
Recall that by $y$ we denote a slice-torus invariant, as defined in the introduction.
\thmD*
\begin{proof}
Let $g$ denote the genus of $\Sigma$.
Assume toward a contradiction that $\Sigma$ is not strongly quasipositive.
Then, by \cref{cor:thm2viagraph}, $\Sigma$ contains a homologically non-trivial unknot $U$ whose framing induced by $\Sigma$ is some non-negative integer $k$.
Choose a closed disk $D\subset S^3$ such that $D$ intersects $\Sigma$ transversely in a proper arc $I\subset \Sigma$ that lies in the interior of $D$ and, in $\Sigma$, $I$ intersects $U$ transversely in exactly one point.
Let $\Sigma'\subset S^3$ be the surface obtained from $\Sigma$ by a $\pm1/k$ surgery along $\partial D$, where the sign is chosen such that $U\subset \Sigma'$ has framing $0$ induced by~$\Sigma'$.
Note that one gets from the knot $K = \partial \Sigma$ to the knot $J \coloneqq \partial \Sigma'$ by $k$ crossing changes from negative to positive,
and so $y(K) \leq y(J)$.
Hence 
ambient surgery in $B^4$ of $\Sigma'$ along $U$ (i.e.\ replacing an annular neighborhood of $U$ in $\Sigma'$ by two discs properly embedded in~$B^4$) produces a slice surface $F$ of~$J$ with $\text{genus}(F) = \text{genus}(\Sigma') - 1 = g - 1$.
It follows that $y(K) \leq y(J) \leq g - 1$, contradicting the assumption $y(K) = g$.
\end{proof}
\corE*
\begin{proof}
(1) $\Rightarrow$ (2) holds by definition; (2) $\Rightarrow$ (3) is implied by $\mathrm{writhe}(\beta) - n +1= 2g(K)$
for a strongly quasipositive braid $\beta$ on $n$ strands whose closure is $K$; (3) $\Rightarrow$ (4) follows since 
$\frac{\mathrm{sl}(K)+1}{2} \leq y (K)\leq g(K)$ holds for all knots $K$; and (4) $\Rightarrow$ (1) 
is immediate from \Cref{prop:s=cang=>sqp}.
\end{proof}

\bibliographystyle{myamsalpha}
\bibliography{References}

\end{document}